\documentclass[reqno]{amsart}
\usepackage[latin1]{inputenc}
\usepackage{color}
\usepackage{amssymb}
\usepackage{hyperref}
\usepackage[all]{xy}
\usepackage{enumitem}
\newtheorem{thm}{Theorem}[section]
\newtheorem{cor}[thm]{Corollary}
\newtheorem{lem}[thm]{Lemma}
\newtheorem{prop}[thm]{Proposition}
\theoremstyle{definition}
\newtheorem{defn}[thm]{Definition}
\newtheorem{rem}[thm]{Remark}

\newtheorem{example}[thm]{Example}
\newtheorem{question}[thm]{Question}
\numberwithin{equation}{section}

\def\epsilon{\varepsilon}

\newcommand{\erre}{\mathbb R}

\newcommand{\fbl}{\operatorname{FBL}}
\newcommand{\R}{\mathbb{R}}
\newcommand{\N}{\mathbb{N}}

\renewcommand{\leq}{\ensuremath{\leqslant}}
\renewcommand{\le}{\ensuremath{\leqslant}}
\renewcommand{\geq}{\ensuremath{\geqslant}}
\renewcommand{\ge}{\ensuremath{\geqslant}}

\title[Banach lattices of positively homogeneous functions]{Banach lattices of positively homogeneous functions induced by a Banach space}

\author[N.J.~Laustsen]{Niels Jakob Laustsen}
\address{Department of Mathematics and Statistics, Fylde College, Lancaster University, Lancaster,
LA1 4YF, United Kingdom}
\email{n.laustsen@lancaster.ac.uk}

\author[P.~Tradacete]{Pedro Tradacete}
\address{Instituto de Ciencias Matem\'aticas (CSIC-UAM-UC3M-UCM)\\
Consejo Superior de Investigaciones Cient\'ificas\\
C/ Nicol\'as Cabrera, 13--15, Campus de Cantoblanco UAM\\
28049 Madrid, Spain.}
\email{pedro.tradacete@icmat.es}


\begin{document}

\setcounter{tocdepth}{1}
\date{\today}

\begin{abstract}
Motivated by the construction of the free Banach lattice generated by a Banach space, we introduce and study several vector and Banach lattices of positively homogeneous functions defined on the dual of a Banach space~$E$. The relations between these lattices allow us to give multiple characterizations of when the underlying Banach space~$E$ is finite-dimensional and when it is reflexive. Furthermore, we show that lattice homomorphisms between free Banach lattices are always composition operators, and study how these operators behave on the scale of lattices of positively homogeneous functions.
\end{abstract}

\maketitle

\section{Introduction}
\noindent
From the point of view of Banach space theory, Banach lattices constitute an important class which includes many of the classical function spaces arising in analysis. The order structure of a Banach lattice is deeply tied to the geometry of its norm, and much research has been carried out to elucidate this relationship.

A common approach in the literature is to use a Banach space~$E$ to construct, or ``induce'', a Banach lattice. We shall study a collection of Banach and vector lattices obtained in this way, consisting of positively homogeneous functions defined on the dual space of~$E$, with the aim of understanding how key geometric properties of~$E$ are expressed in the order structure of and relationship between the Banach lattices it induces.

Our motivation comes from the study of ``free Banach lattices''. This notion arose in the  
work of de~Pag\-ter and Wick\-stead~\cite{dePW}, who defined the free Banach lattice over a set.  Sub\-sequent\-ly, Avil\'{e}s, Rodr\'{\i}guez and Tradacete~\cite{ART} defined the free Banach lattice~$\fbl[E]$ generated by a Banach space~$E$ and observed that this object generalizes that of de~Pag\-ter and Wick\-stead in the sense that, for a set~$\Gamma$, $\fbl[\ell_1(\Gamma)]$~is the free Banach lattice over~$\Gamma$.

As shown in \cite{AMR:20, AMR:21, dePW}, free Banach lattices are closely related to the notion of projectivity. They also played a key role in the recent construction of push-outs in the category of Banach lattices \cite{AT},  and they have been instrumental in resolving a number of open problems about the structure of Banach lattices, including the following:
\begin{itemize}
    \item A Banach lattice may be lattice weakly compactly generated, but not embed into any weakly compactly generated Banach space \cite[Corollary~5.5]{ART}; this answered a question of Diestel.
\item A lattice homomorphism between Banach lattices need not attain its operator norm \cite{Norm-attaining}.
\item A closed, infinite-dimensional subspace of a Banach lattice need not contain any bi\-basic sequences  \cite[Theorem 7.5]{OTTT}. 
\end{itemize}
Other recent results concern the relationship between a Banach space~$E$ and the free Banach lattice~$\fbl[E]$ it generates. Let us briefly mention a few of these: 
\begin{itemize}
    \item $E$ is separable if and only if $\fbl[E]$ is separable \cite[Theorem~3.2]{ART}, if and only if $\fbl[E]$ has a quasi-interior point \cite[Proposition~9.4]{OTTT}.
\item $\fbl[E]^*$ is order continuous precisely when $E$ does not contain any complemented copies of $\ell_1$  \cite[Theorem 9.20]{OTTT}.
\item $\fbl[E]$ satisfies an upper $p$-estimate if and only if the identity operator on~$E^*$ is $(p^*,1)$-summing, where~$p^*$ is the conjugate exponent of~$p$  \cite[Theorem 9.21]{OTTT}. 
\item $\fbl[E]$ always satisfies the countable chain condition \cite{APR}.
\end{itemize}

Formally speaking, the \emph{free Banach lattice generated by a Banach space~$E$} is a Banach lattice~$\fbl[E]$ together with a linear isometry $\delta^E\colon E\rightarrow \fbl[E]$ such that, for every Banach lattice~$X$ and every bounded linear operator $T\colon E\rightarrow X$, there exists a unique lattice homomorphism
$\hat T:\fbl[E]\rightarrow X$ such that the diagram 
\[ 	\xymatrix{\fbl[E]\ar^{\hat{T}}[drr]&&\\
	E\ar^{\delta^E}[u]\ar[rr]^T&&X } \]
is commutative, and $\|\hat T\|=\|T\|$. This construction defines a covariant functor from the category of Banach spaces and bounded linear operators into the category of Banach lattices and lattice homomorphisms.

Standard arguments show that if the free Banach lattice generated by a Banach space~$E$ exists, then it is unique up to isometric lattice isomorphism, so the significance of~\cite{ART} is the proof   that~$\fbl[E]$ exists. Constructing free vector lattices is quite straightforward (see \cite{Baker, Bleier}), but the method does not immediately carry over to Banach lattices (although it is possible to adapt it, as shown in~\cite{T19}).  Instead, the authors of~\cite{ART} took a different route, constructing~$\fbl[E]$ explicitly as a Banach lattice of functions. This additional feature has turned out to be very useful and is a cornerstone of our work. Let us therefore begin with a review of the construction.

Take a Banach space~$E$, and denote by~$H[E]$  the linear subspace of $\erre^{E^*}$
consisting of all positively homogeneous functions, where we recall that a function $f\colon E^*\to\R$ is positively homogeneous if $f(\lambda x^*)=\lambda f(x^*)$ for every $x^*\in E^*$ and \mbox{$\lambda\in[0,\infty)$}. 
For $f\in H[E]$, define
\begin{equation}\label{Defn:fblNorm}
    \|f\|_{\fbl[E]} =
	\sup\biggl\{\sum_{j=1}^n |f(x_j^\ast)| : n\in\mathbb{N},\, (x_j^*)_{j=1}^n\subset E^*,\,\sup_{\lVert x\rVert\le1} \sum_{j=1}^n |x_j^\ast(x)|\leq 1\biggr\}. 
\end{equation}
Then $H^1[E] =\{f\in H[E]: \|f\|_{\fbl[E]}<\infty\}$ is a Banach lattice with respect to the norm~$\|\cdot\|_{\fbl[E]}$
and the pointwise vector lattice operations. Now, for \mbox{$x\in E$}, the map $\delta_x\colon E^*\rightarrow \mathbb R$ defined by $\delta_x(x^\ast) = x^\ast(x)$ belongs to~$H^1[E]$ with $\|\delta_x\|_{\fbl[E]}=\|x\|_E$, and \cite[Theorem~2.5]{ART} shows that the free Banach lattice $\fbl[E]$ can be realized as the closed sublattice of $H^1[E]$ generated by $\{\delta_x:x\in E\}$, together with the linear isometry $\delta^E\colon E\rightarrow \fbl[E]$ given by $\delta^E(x) =\delta_x$ for $x\in E$.

This construction was generalized in~\cite{JLTTT} to produce the \emph{free $p$-convex Banach lattice} generated by a Banach space~$E$ for $1\le p\le\infty$. To state the definition, 
set 
\begin{equation}\label{Defn:FBLpNorm}
   \|f\|_{\fbl^p[E]} =
	\sup\biggl\{\Bigl(\sum_{j=1}^n |f(x_j^\ast)|^p\Bigr)^{\frac{1}{p}} : n\in\mathbb N,\, (x_j^*)_{j=1}^n\subset E^*,\,\lVert (x_j^*)_{j=1}^n\rVert_{p,{\normalfont{\text{weak}}}}\leq 1\biggr\}
\end{equation}
for $f\in H[E]$, where 
\begin{equation}\label{Eq:weakpsumnorm}
    \lVert (x_j^*)_{j=1}^n\rVert_{p,{\normalfont{\text{weak}}}} = \sup_{\lVert x\rVert\le 1}\biggl(\sum_{j=1}^n \lvert x^*_j(x)\rvert^p\biggr)^{\frac1p} 
\end{equation} 
denotes the weak $p$-summing norm of the $n$-tuple $(x_j^*)_{j=1}^n$ in~$E^*$. (As usual, expressions of the form $\bigl(\sum_{j=1}^n \lvert t_j\rvert^p\bigr)^{\frac1p}$ should be interpreted as $\max_{1\le j\le n}\lvert t_j\rvert$ for $p=\infty$.) 
Note that for $p=1$, \eqref{Eq:weakpsumnorm} can  equivalently be written as
\begin{equation}\label{Eq:weakpsumnorm:1}
 \lVert (x_j^*)_{j=1}^n\rVert_{1,{\normalfont{\text{weak}}}} = \sup_{\varepsilon_i=\pm 1}\Bigl\|\sum_{j=1}^n \varepsilon_j x_j^*\Bigr\|.
\end{equation}
Then $H^p[E]=\{f\in H[E]: \|f\|_{\fbl^p[E]}<\infty\}$
is a Banach lattice as before, and \cite[Theorem~6.1]{JLTTT} shows that the closed sublattice of~$H^p[E]$ generated by the set $\{\delta_x:x\in E\}$ together with the linear isometry $\delta^E\colon E\rightarrow \fbl^p[E]$ given by $\delta^E(x) = \delta_x$ for $x\in E$ is the free $p$-convex Banach lattice generated by $E$. 

We observe that $\fbl[E] = \fbl^1[E]$ because~\eqref{Defn:FBLpNorm} reduces to~\eqref{Defn:fblNorm} for $p=1$.
A~less obvious result, shown in \cite[Proposition~2.2]{OTTT}, is that $\fbl^\infty[E]$ can be identified with the space of positively homogeneous weak-$*$ continuous functions defined on the dual unit ball. 

The role of positive homogeneity in these constructions is not coincidental, but a consequence of the fact that the functions $\delta_x$ for $x\in E$ are our building blocks. They are linear, and hence lattice combinations of them will always be positively homogeneous. We remark that earlier work on Banach lattices of homogeneous functions can be found in \cite{Goullet, Rogalski70, Rogalski71, Wickstead}.

The purpose of the present paper is to investigate~$H^p[E]$ for $1\le p<\infty$ in its own right, together with a number of sublattices and ideals lying between~$\fbl^p[E]$ and~$H^p[E]$.  This will in particular lead to characterizations of the cases where the underlying Banach space~$E$ is finite-dimen\-sional (Theorem \ref{L5.3}) and reflexive (Theorem \ref{t:ref}). 

In Section \ref{S:4}, we will see that in the infinite-dimensional situation, one should be able to provide examples of weak* continuous functions in $H^p[E]$ that are not in $\fbl^p[E]$. Surprisingly, our proof requires the existence of a separable quotient of the underlying Banach space $E$. Although this might be just an artefact of the proof, in any case it provides a fairly general condition that allows us to distinguish $\fbl^p[E]$ within the weak* continuous functions of $H^p[E]$.

Finally, in Section~\ref{section:isomorphism} we study lattice homomorphisms between free Banach lattices. Our main result (Theorem~\ref{t:latticeiso}) states that  every lattice homomorphism $T\colon\fbl^p[E]\rightarrow \fbl^p[F]$, where $1\le p<\infty$ and~$E$ and~$F$ are Banach spaces, can be expressed as a composition operator
\[ Tf=f\circ\Phi_T\qquad (f\in \fbl^p[E]) \] 
for a unique map $\Phi_T\colon F^*\rightarrow E^*$, which is positively homogeneous and maps weakly summable sequences in~$F^*$ to weakly summable sequences in~$E^*$, and whose re\-stric\-tion to the closed unit ball of~$F^*$ is weak* continuous. We explore the consequences and applications of this result, including the im\-pact of the map~$\Phi_T$ on the other Banach lattices under investigation.

\section{Banach lattices of positively homogeneous functions defined on a dual Banach space}\label{S:2}

\noindent Take $p\in[1,\infty]$, and let~$E$ be a Banach space, always over the real field. (We refer to \cite{dHT} for the construction and study of free complex Banach lattices over a complex Banach space.) 
The Banach lattice $(H^p[E],\lVert\,\cdot\,\rVert_{\fbl^p[E]})$ plays a central role in the definition of~$\fbl^p[E]$, being the ambient lattice that~$\fbl^p[E]$ is defined as a sublattice of. We shall consider the following vector lattices and ideals, several of which were already mentioned in the Introduction; for ease of reference, we include them here:\label{Eq:latticelist}
\begin{align*}
H[E] &= \{ f\in\R^{E^*} : f\ \text{is positively homogeneous}\}\\
\operatorname{lat}\delta[E] &= \text{the sublattice of}\ H[E]\ \text{generated by}\ \{\delta_x : x\in E\}\\
I[E] &= \text{the ideal of}\ H[E]\ \text{generated by}\ \{\delta_x : x\in E\}\\
H^p[E] &= \{ f\in H[E] : \lVert f\rVert_{\fbl^p[E]}<\infty\}\\
\fbl^p[E] &= \text{the closed sublattice of}\ H^p[E]\ \text{generated by}\ \{\delta_x : x\in E\}\\
H_{w^*}[E] &= \{ f\in H[E] : f{\upharpoonright_{B_{E^*}}}\ \text{is weak*-continuous}\}\\
I_{w^*}[E] &= I[E]\cap H_{w^*}[E]\\
H_{w^*}^p[E] &= H^p[E]\cap H_{w^*}[E]\\
J_{w^*}^p[E] &= \text{the ideal of}\ H^p[E]\ \text{generated by}\ H_{w^*}^p[E]\\
M_{w^*,0}^p[E] &= \{ f\in H^p[E] : f{\upharpoonright_{B_{E^*}}}\ \text{is weak*-continuous at}\ 0\},
\end{align*}
where~$B_{E^*}$ denotes the closed unit ball of~$E^*$. 
Note the change of terminology here: what we call $H^1[E]$ was originally denoted~$H_0[E]$ in~\cite{ART}.
As already stated in the Introduction, $H^p[E]$ is a sublattice of~$H[E]$ and a Banach lattice with respect to the norm \mbox{$\|\cdot\|_{\fbl^p[E]}$} defined by~\eqref{Defn:FBLpNorm}. 
 
For $f\in H[E]$, we write
\begin{equation}\label{Eq:uniformnorm} 
    \lVert f{\upharpoonright_{B_{E^*}}}\rVert_\infty=\sup\bigl\{\lvert f(x^*)\rvert : x^*\in B_{E^*}\bigr\}\in[0,\infty]
\end{equation}
for the uniform norm of the restriction of~$f$ to~$B_{E^*}$.
Then \cite[Proposition~2.2]{OTTT} shows that $(\fbl^\infty[E],\|\cdot\|_{\fbl^\infty[E]})$ is isometrically lattice isomorphic to $(H_{w^*}[E],\|\cdot\|_\infty)$, as previously mentioned.
For that reason, we shall restrict our attention to the case $p<\infty$ in the remainder of this paper.   
We begin with some basic general ob\-ser\-va\-tions.

\begin{lem}\label{L5.1}
  Let $E$ be a Banach space, and take $1\le p<\infty$. Then:
\begin{enumerate}[label={\normalfont{(\roman*)}}]
\item\label{L5.1_0} $I[E] = \{ f\in H[E] : \lvert f\rvert\le \sum_{j=1}^n \lvert \delta_{x_j}\rvert\ \text{for some}\ n\in\N\ \text{and}\ x_1,\ldots,x_n\in E\}$\\
\mbox{}$\phantom{I[E]} = \{ f\in H[E] : \lvert f\rvert\le \bigvee_{j=1}^n \lvert \delta_{x_j}\rvert\ \text{for some}\ n\in\N\ \text{and}\ x_1,\ldots,x_n\in E\}$.
\item\label{L5.1ii} $H_{w^*}[E]$ is a sublattice of $H[E]$.
\item\label{L5.1vi'} The norm~\eqref{Defn:FBLpNorm} dominates the norm~\eqref{Eq:uniformnorm}{\normalfont{;}} that is, $\|f\|_{\fbl^p[E]}\ge\lVert f{\upharpoonright_{B_{E^*}}}\rVert_\infty$ for every $f\in H^p[E]$.
\item\label{L5.1vi} $H_{w^*}[E]$ is closed in the norm~$\lVert\,\cdot\,\rVert_\infty$ 
defined by~\eqref{Eq:uniformnorm} in the following precise sense: Suppose that $(f_n)$ is a sequence in $H_{w^*}[E]$ which converges uniformly on~$B_{E^*}$ to a function $f\in H[E]$. Then $f\in H_{w^*}[E]$.
\item\label{L5.1ii'} $H_{w^*}^p[E]$ is a closed sublattice of the Banach lattice $(H^p[E],\lVert\,\cdot\,\rVert_{\fbl^p[E]})$.
\item\label{L5.1iii} $\fbl^p[E]\subseteq H_{w^*}[E]$.
\item\label{lemma:Mw*0i} $I[E]\subseteq J_{w^*}^p[E]$.
\item\label{lemma:Mw*0iii} $\overline{I_{w^*}[E]}\subseteq H_{w^*}^p[E]$.
\item\label{lemma:Mw*0ii} $M_{w^*,0}^p[E]$ is a closed ideal of~$H^p[E]$ containing $\overline{J_{w^*}^p[E]}$.
\end{enumerate}
\end{lem}

\begin{proof}  \ref{L5.1_0}. The first equality follows from the fact  that the set on the right-hand side is the smallest solid subspace of~$H[E]$ which contains $\delta_x$ for every $x\in E$.  The second equality is immediate from the inequalities
  \[ \bigvee_{j=1}^n \lvert \delta_{x_j}\rvert\le \sum_{j=1}^n \lvert \delta_{x_j}\rvert\le \bigvee_{j=1}^n \lvert \delta_{nx_j}\rvert\qquad (n\in\N,\, x_1,\ldots,x_n\in E).\]
  
  \ref{L5.1ii}. This is a consequence of the fact that the lattice operations in $H[E]$ are defined pointwise.

 \ref{L5.1vi'}. This is clear because $\lVert x^*\rVert_{p,{\normalfont{\text{weak}}}} = \lVert x^*\rVert$ for each $x^*\in E^*$.  
 
 \ref{L5.1vi}. This follows from the general result that the uniform limit of a sequence of continuous functions defined on a compact Hausdorff space is continuous.

  \ref{L5.1ii'}. Clause~\ref{L5.1ii} implies that $H_{w^*}^p[E]$ is a sublattice of~$H^p[E]$. To see that it is closed,   suppose that $(f_n)$ is a sequence in    $H_{w^*}^p[E]$ which converges to $f\in H^p[E]$ in the norm $\lVert\,\cdot\,\rVert_{\fbl^p[E]}$. Then
    $(f_n)$ converges uniformly on~$B_{E^*}$ to~$f$ by~\ref{L5.1vi'},
 so~\ref{L5.1vi} implies that $f\in H_{w^*}[E]$.

 \ref{L5.1iii}. This is proved for $p=1$ in \cite[Lemma~4.10]{ART}, and the argument given there\-in carries over verbatim to $p>1$. Indeed, \ref{L5.1ii'} implies that $S:=\fbl^p[E]\cap H_{w^*}[E]$ is a closed sublattice of~$(\fbl^p[E],\lVert\,\cdot\,\rVert_{\fbl^p[E]})$. Clearly $\delta_x\in S$ for every $x\in E$, so $S=\fbl^p[E]$; that is, $\fbl^p[E]\subseteq H_{w^*}[E]$.  

\ref{lemma:Mw*0i}. We begin by showing that $I[E]\subseteq H^p[E]$. Take $f\in I[E]$.  By~\ref{L5.1_0}, we can find $n\in\N$ and $x_1,\ldots,x_n\in E$ such that $\lvert f\rvert\le\sum_{k=1}^n\lvert\delta_{x_k}\rvert$. Suppose that $m\in\N$ and $x_1^*,\ldots,x^*_m\in E^*$ with $\lVert (x_j^*)_{j=1}^m\rVert_{p,{\normalfont{\text{weak}}}}\le 1$. Then $\sum_{j=1}^m \lvert x_j^*(x_k)\rvert^p\le \lVert x_k\rVert^p$ for each $k\in\{1,\ldots,n\}$,  and consequently
 \begin{align*} \biggl(\sum_{j=1}^m \lvert f(x_j^*)\rvert^p\biggr)^{\frac1p}\le \biggl(\sum_{j=1}^m\Bigl(\sum_{k=1}^n \lvert x_j^*(x_k)\rvert\Bigr)^p\biggr)^{\frac1p}\le \sum_{k=1}^n\biggl(\sum_{j=1}^m \lvert x_j^*(x_k)\rvert^p\biggr)^{\frac1p}\le \sum_{k=1}^n \lVert x_k\rVert, \end{align*}
 where the second inequality is simply the statement that the norm of the sum of~$n$ vectors in~$\ell_p^m$ is dominated by the sum of their norms. 
 We conclude that $\lVert f\rVert_{\fbl^p[E]}\le \sum_{k=1}^n\lVert x_k\rVert<\infty$, and therefore $f\in H^p[E]$, as desired.
 
This implies that~$I[E]$ is the ideal of~$H^p[E]$ generated by~$\delta_x$ for $x\in E$. These generators belong to $H_{w^*}^p[E]$ (as we already observed in~\ref{L5.1iii} above), and hence $I[E]\subseteq J_{w^*}^p[E]$.

\ref{lemma:Mw*0iii}. We have just seen that $I[E]\subseteq H^p[E]$, so $I_{w^*}[E]\subseteq H_{w^*}^p[E]$. Now the con\-clu\-sion follows from the fact that $H_{w^*}^p[E]$ is closed by~\ref{L5.1ii'}.  
 
\ref{lemma:Mw*0ii}. Suppose that $f$  belongs to the closure of the ideal of~$H^p[E]$ generated by~$M_{w^*,0}^p[E]$, and take a weak*-null net $(x^*_\alpha)$ in $B_{E^*}$. The set $M_{w^*,0}^p[E]$ is a sub\-lattice of $H[E]$ because the vector lattice operations are defined pointwise. Therefore, for each $\varepsilon>0$, we can find $g\in H^p[E]$ and $h\in M_{w^*,0}^p[E]$ such that 
$\lVert f-g\rVert_{\fbl^p[E]}\le\varepsilon$ and $\lvert g\rvert\le \lvert h\rvert$. Since $h(0)=0$ by positive homogeneity, we can choose $\alpha_0$ such that $\lvert h(x_\alpha^*)\rvert\le \varepsilon$ for every $\alpha\ge \alpha_0$. Then we have 
\begin{align*}
   \lvert f(x^*_\alpha)-f(0)\rvert = \lvert f(x^*_\alpha)\rvert&\le \lvert f(x_\alpha^*)-g(x^*_\alpha)\rvert + \lvert g(x_\alpha^*)\rvert\\
   &\le   \lVert f-g\rVert_{\fbl^p[E]} + \lvert h(x_\alpha^*)\rvert\le 2\varepsilon\qquad (\alpha\ge \alpha_0), 
\end{align*} 
which shows that $f\in M_{w^*,0}^p[E]$, and therefore $M_{w^*,0}^p[E]$ is a closed ideal of~$H^p[E]$. 

This implies that  $\overline{J_{w^*}^p[E]}\subseteq M_{w^*,0}^p[E]$ because $H_{w^*}^p[E]\subseteq M_{w^*,0}^p[E]$.
 \end{proof}

The following two chains of inclusions summarize the relationship among the sets defined above:
\begin{align}\label{E:inclusions1}
\operatorname{lat}\delta[E]&\subseteq \fbl^p[E]\subseteq \overline{I_{w^*}[E]}\subseteq H_{w^*}^p[E]\notag\\ &\subseteq J_{w^*}^p[E]\subseteq \overline{J_{w^*}^p[E]}\subseteq M_{w^*,0}^p[E]\subseteq H^p[E],
\end{align}
where~$\overline{S}$ denotes the closure of a set~$S$ with respect to the norm~$\lVert\,\cdot\,\rVert_{\fbl^p[E]}$, and 
\begin{equation}\label{E:inclusions2} \operatorname{lat}\delta[E]\subseteq I_{w^*}[E]\subseteq I[E]\subseteq J_{w^*}^p[E]. \end{equation}
In the next three sections we shall examine when two (or more) of these lattices are equal. 

\section{Characterizations of finite-dimensionality}\label{S:3}
\noindent
In this section we study the chains of lattices~\eqref{E:inclusions1} and~\eqref{E:inclusions2} when the Banach space~$E$ is finite-dimensional. It turns out that in this case the lattices will effec\-tive\-ly reduce  to two options, either~$C(S_{E^*})$ or~$\ell_\infty(S_{E^*})$, where $S_{E^*} = \{ x^*\in E^* : \lVert x^*\rVert =1\}$ denotes the dual unit sphere; see Proposition~\ref{prop:findimcase} and Corollary~\ref{Cor:3.5} for details. 

We begin by showing that in dimensions~$0$ and~$1$, even more is true: all the lattices defined on page~\pageref{Eq:latticelist} are equal, and this happens only in these two dimensions. Perhaps more surprisingly, this is equivalent to the second inclusion $I_{w^*}[E]\subseteq I[E]$ in~\eqref{E:inclusions2} being an equality. 

\begin{lem}\label{L5.2}
  Let $E$ be a Banach space, and take $1\le p<\infty$. Then the following conditions are equivalent:
\begin{enumerate}[label={\normalfont{(\alph*)}}]
\item\label{L5.2c} $\dim E\le 1;$
\item\label{L5.2a} $\operatorname{lat}\delta[E] = H[E];$
\item\label{L5.2e} $H^p[E]\subseteq H_{w^*}[E];$  
\item\label{L5.2d} $H_{w^*}^p[E] = J_{w^*}^p[E];$  
\item\label{L5.2b} $I_{w^*}[E] = I[E]$.
\end{enumerate}
\end{lem}

\begin{proof}
  \ref{L5.2c}$\Rightarrow$\ref{L5.2a}. We have $\operatorname{lat}\delta[E] = H[E]$ for $E=\{0\}$ because $H[E] =\{0\}$ in this case.

  Now suppose that $E=\R$, in which case $E^*=\R$, and the duality is given by $\delta_x(y)=\langle x,y\rangle = xy$ for $x,y\in\R$. Take $f\in H[\R]$, and suppose that $f(1)\ge -f(-1)$. Then, for each $y\in\R$, we have 
\begin{align*}
     (\delta_{f(1)}\vee \delta_{-f(-1)})(y) &= \delta_{f(1)}(y)\vee \delta_{-f(-1)}(y) = f(1)y\vee (-f(-1)y)\\ &= \left\{ \begin{array}{ll} f(1)y &\text{for}\ y\ge 0\\ -f(-1)y &\text{for}\ y<0\end{array}\right\} = f(y) 
\end{align*}
by positive homogeneity. Hence $f=\delta_{f(1)}\vee \delta_{-f(-1)}$. A similar argument shows that  $f=\delta_{f(1)}\wedge \delta_{-f(-1)}$ when $f(1)\le -f(-1)$, and so $f\in\operatorname{lat}\delta[\R]$ in both cases. 

  We have \ref{L5.2a}$\Rightarrow$\ref{L5.2e} because the inclusions $H^p[E]\subseteq H[E]$ and $\operatorname{lat}\delta[E]\subseteq H_{w^*}[E]$ hold true in general. 

\ref{L5.2e}$\Rightarrow$\ref{L5.2d}. Suppose that $H^p[E]\subseteq H_{w^*}[E]$. Then we have $H_{w^*}^p[E] = H^p[E]$, from which~\ref{L5.2d} follows. 

\ref{L5.2d}$\Rightarrow$\ref{L5.2b}. Suppose that $H_{w^*}^p[E] = J_{w^*}^p[E]$. Then 
Lemma~\ref{L5.1}\ref{lemma:Mw*0i} implies that $I[E]\subseteq H_{w^*}^p[E]\subseteq H_{w^*}[E]$.

Finally, we prove that \ref{L5.2b}$\Rightarrow$\ref{L5.2c} by contraposition. Suppose that $\dim E\ge 2$, and  choose two linearly independent functionals~$x^*_0,y^*_0\in E^*$, each having norm~$\frac12$. Then $(x_0^* + \frac1n y_0^*)_{n\in\N}$ is a sequence in~$B_{E^*}$ which converges to~$x_0^*$ in norm and therefore also in the weak*-topology.
  Define $f\colon E^*\to\R$ by
  \[ f(x^*) = \begin{cases} t\ &\text{if}\ x^* = tx_0^*\ \text{for some}\ t>0,\\ 0\ &\text{otherwise.} \end{cases} \]
  It is easy to see that $f$ is positively homogeneous. Choose $x_0\in E$ such that $x_0^*(x_0)=1$. Then
    $\lvert f\rvert\le \lvert\delta_{x_0}\rvert$, so $f\in I[E]$. On the other hand, $f(x_0^* + \frac1n y_0^*)=0$ for each $n\in\N$ by definition, but $f(x_0^*) =1$, so $f{\upharpoonright_{B_{E^*}}}$ is discontinuous at~$x_0^*$, and therefore $f\notin H_{w^*}[E]$.
\end{proof}

 Let $E$ be a non-zero Banach space. 
 The restriction mapping $R\colon f\mapsto f{\upharpoonright_{S_{E^*}}}$ defines a lattice isomorphism of~$H[E]$ onto the vector lattice of functions $S_{E^*}\to\R$ because every function $S_{E^*}\to\R$ extends uniquely to a positively homogeneous func\-tion~$E^*\to\R$. When $E$ is finite-dimensional, we can use restrictions of this lattice isomorphism~$R$ to identify the Banach lattices $\fbl^p[E]$ and $H^p[E]$ explicitly. 

\begin{prop}\label{prop:findimcase} Let $E$ be a non-zero, finite-dimensional Banach space. Then, for every $1\le p<\infty$,  
\[ R(\fbl^p[E]) = R(H_{w^*}[E]) = C(S_{E^*})\quad\text{and}\quad R(J_{w^*}^p[E]) = R(H^p[E]) = \ell_\infty(S_{E^*}), \] 
where $\ell_\infty(S_{E^*})$ denotes the Banach lattice of bounded functions $S_{E^*}\to\R$.
\end{prop}

\begin{proof} Lemma~\ref{L5.1}\ref{L5.1iii} implies that  
$R(\fbl^p[E])\subseteq R(H_{w^*}[E])$, and the inclusion $R(H_{w^*}[E])\subseteq C(S_{E^*})$ holds because the weak* and norm topologies on~$E^*$ coincide. Therefore, to prove the left-hand identity, it only remains to verify that $C(S_{E^*})\subseteq R(\fbl^p[E])$, which can be done by arguing as in the proof of \cite[Proposition~5.3]{dePW}; see also  \cite[Corollary~2.9(iii)]{ART} for the connection between  de~Pagter and Wickstead's notion of the free Banach lattice over a set studied in~\cite{dePW} and  the free Banach lattice over a Banach space. 

To prove the right-hand identity, we begin by observing that $R(J_{w^*}^p[E])\subseteq R(H^p[E])$ because $J_{w^*}^p[E]\subseteq H^p[E]$ by definition, and $R(H^p[E])\subseteq\ell_\infty(S_{E^*})$ because $\lvert f(x^*)\rvert\le \lVert f\rVert_{\fbl^p[E]}$ for every $x^*\in B_{E^*}$ and $f\in H^p[E]$. Finally, to verify that $\ell_\infty(S_{E^*})\subseteq R(J_{w^*}^p[E])$, take $f_0\in \ell_\infty(S_{E^*})$. In view of the remarks we made before the statement of the result, we can find $f\in H[E]$ such that $Rf = f_0$. Let $\nu\in H[E]$ be the ``norm function'' defined by $\nu(x^*) = \lVert x^*\rVert$ for every $x^*\in E^*$. Then $R\nu$ is the constant function 1, which obviously belongs to~$C(S_{E^*})$, so $\nu\in\fbl^p[E]\subseteq H_{w^*}^p[E]$ by the first part of the proof and Lemma~\ref{L5.1}\ref{L5.1iii}. Since~$f$ is positively homogeneous, we have $\lvert f\rvert\le \lVert f{\upharpoonright_{S_{E^*}}}\rVert_\infty\, \nu$, which shows that $f\in J_{w^*}^p[E]$, and the conclusion follows. 
\end{proof}

\begin{rem}
Note that the identities in Proposition \ref{prop:findimcase} are up to equivalence of norms. In fact, one can check the constants in this equivalence increase as the dimension of the underlying Banach space $E$ does.    
\end{rem}

The following general observation will be useful for us. 

\begin{lem}\label{Lemma:supinHw*}
  Let $(x_n)$ be a sequence in a Banach space~$E$ with $\|x_n\|\underset{n\to\infty}\longrightarrow0$. Then \[ \bigvee_{n\in\N} \lvert\delta_{x_n}\rvert\in H_{w^*}[E]. \]
\end{lem}

\begin{proof}
Set $f = \bigvee_{n\in\N} \lvert\delta_{x_n}\rvert$, which is clearly a positively homogeneous function, and consider the functions $f_m =  \bigvee_{n=1}^m \lvert\delta_{x_n}\rvert\in H_{w^*}[E]$ for $m\in\N$. They satisfy
\begin{align*} \lVert (f-f_m){\upharpoonright_{B_{E^*}}}\rVert_\infty &= \sup_{x^*\in B_{E^*}} \biggl|  \bigvee_{n=1}^\infty\lvert x^*(x_n)\rvert - \bigvee_{n=1}^m\lvert x^*(x_n)\rvert\biggr|\le \sup_{x^*\in B_{E^*}} \bigvee_{n=m+1}^\infty \lvert x^*(x_n)\rvert\\  &= \bigvee_{n=m+1}^\infty \lVert x_n\rVert\underset{m\to\infty}\longrightarrow 0, \end{align*}
so $f\in H_{w^*}[E]$ by Lemma~\ref{L5.1}\ref{L5.1vi}. 
\end{proof}

\begin{thm}\label{L5.3}
  Let $E$ be a Banach space, and take $1\le p<\infty$.  Then the following conditions are equivalent:
  \begin{enumerate}[label={\normalfont{(\alph*)}}]
  \item\label{L5.3a} $E$ is finite-dimensional;
  \item\label{L5.3b} $\fbl^p[E] = H_{w^*}[E];$
  \item\label{L5.3c} $H_{w^*}[E]\subseteq H^p[E];$
  \item\label{L5.3h}  $I[E] = H^p[E];$
  \item\label{L5.3j}  $I[E] = J_{w^*}^p[E];$
  \item\label{L5.3i} $I[E]$ is closed in~$H^p[E];$
  \item\label{L5.3g}  $\fbl^p[E]\subseteq I[E];$
  \item\label{L5.3d}  $I_{w^*}[E] = H_{w^*}[E];$
  \item\label{L5.3e} $I_{w^*}[E]$ is closed in~$H^p[E];$
  \item\label{L5.3f}  $\fbl^p[E]\subseteq I_{w^*}[E]$.
\end{enumerate}
\end{thm}

\begin{proof} The proof has three parts: first we show that conditions~\ref{L5.3a}--\ref{L5.3c} are equivalent, then we show that conditions~\ref{L5.3a} and \ref{L5.3h}--\ref{L5.3g} are equivalent, and finally we show that conditions \ref{L5.3d}--\ref{L5.3f} are equivalent to the other conditions.

  \ref{L5.3a}$\Rightarrow$\ref{L5.3b}. This follows from the first identity in Proposition~\ref{prop:findimcase}.

  \ref{L5.3b}$\Rightarrow$\ref{L5.3c}. This is obvious because $\fbl^p[E]\subseteq H^p[E]$ by definition.

\ref{L5.3c}$\Rightarrow$\ref{L5.3a}. To prove the  contrapositive, suppose that $E$ is infinite-dimensional. Then, by the weak Dvoretzky--Rogers Theorem \cite[Theorem~2.18]{DJT}, $E^*$ contains a sequence $(x_n^*)$ which is weakly $p$-summable, but not strongly $p$-summable. By scaling this sequence by a suitable constant, we may suppose that
  \[ \sup_{x\in B_E}\sum_{n=1}^\infty|x_n^*(x)|^p\leq1\qquad\text{and}\qquad \sum_{n=1}^\infty\|x_n^*\|^p=\infty. \]
 For each $n\in\mathbb N$, choose
$x_n\in B_E$ such that $x_n^*(x_n)\geq \frac{\|x_n^*\|}{2}$, and
choose a decreasing sequence $(s_n)$ in $(0,1)$ such that 
\[ \lim_{n\to\infty} s_n=0\qquad\text{and}\qquad \sum_{n=1}^\infty s_n^p\|x_n^*\|^p=\infty. \]
Lemma~\ref{Lemma:supinHw*} shows that the function
$f = \bigvee_{n\in\mathbb N} \lvert\delta_{s_nx_n}\rvert$
belongs to $H_{w^*}[E]$. However, $f\notin H^p[E]$ because the fact that $\lVert(x_k^*)_{k=1}^m\rVert_{p,{\normalfont{\text{weak}}}}\le 1$ for every $m\in\N$ implies that
\begin{align*}
\|f\|^p_{\fbl^p[E]} &\geq \sum_{k=1}^m |f(x_k^*)|^p = \sum_{k=1}^m \biggl(\bigvee_{n\in\mathbb N} s_n |x_k^*(x_n)|\biggr)^p\\ &\geq \sum_{k=1}^m s_k^p|x_k^*(x_k)|^p
\geq \frac1{2^p}\sum_{k=1}^m s_k^p\|x_k^*\|^p\underset{m\rightarrow\infty}\longrightarrow\infty.
\end{align*}
Consequently $H_{w^*}[E]\nsubseteq H^p[E]$.

\ref{L5.3a}$\Rightarrow$\ref{L5.3h}.
Suppose that $E$ has dimension~$n\in\N$, and let $(b_j)_{j=1}^n$ be a basis for~$E$ with coordinate functionals $(b_j^*)_{j=1}^n$, which form a basis for~$E^*$. We may suppose that $\lVert b_j^*\rVert = 1$ for each $j\in\{1,\ldots,n\}$. Given $f\in H^p[E]$, the number $t = \lVert f{\upharpoonright_{B_{E^*}}}\rVert_\infty$ is finite by Lemma~\ref{L5.1}\ref{L5.1vi'}, and for each non-zero $x^*\in E^*$, we have
    \begin{align*} \lvert f(x^*)\rvert &= \lVert x^*\rVert\,\biggl| f\Bigl(\frac{x^*}{\lVert x^*\rVert}\Bigr)\biggr|\le  \lVert x^*\rVert\cdot t = \Biggl\|\sum_{j=1}^n x^*(tb_j)b_j^*\biggr\| \le \sum_{j=1}^n\lvert x^*(tb_j)\rvert. 
    \end{align*}
 This inequality is trivially true for $x^*=0$, so $\lvert f\rvert\le  \sum_{j=1}^n\lvert \delta_{tb_j}\rvert$,  and  therefore $f\in I[E]$ by Lemma~\ref{L5.1}\ref{L5.1_0}. This  shows that $H^p[E]\subseteq I[E]$, while the opposite inclusion follows from   Lemma~\ref{L5.1}\ref{lemma:Mw*0i}.

It is clear that \ref{L5.3h}$\Rightarrow$\ref{L5.3i}$\Rightarrow$\ref{L5.3g}. 
We also have \ref{L5.3h}$\Rightarrow$\ref{L5.3j}$\Rightarrow$\ref{L5.3g} because Lemma~\ref{L5.1}\ref{lemma:Mw*0i} and~\ref{L5.1iii} show that the following inclusions hold true in general: 
\[ I[E]\subseteq J_{w^*}^p[E]\subseteq H^p[E]\quad \text{and}\quad  \fbl^p[E]\subseteq H_{w^*}^p[E]\subseteq J_{w^*}^p[E]. \] 

We prove that \ref{L5.3g}$\Rightarrow$\ref{L5.3a} by contraposition. Suppose that $E$ is infinite-dimensional, and take a normalized basic sequence $(b_n)_{n\in\N}$ in~$E$.   Let~$b_n^*\in E^*$ be a Hahn--Banach extension of the $n^{\text{th}}$ coordinate functional of this basic sequence for each $n\in\N$, and define $f = \sum_{n=1}^\infty\lvert\delta_{b_n}\rvert/2^n$, which belongs to~$\fbl^p[E]$ because  the series converges absolutely in the norm~$\lVert\,\cdot\,\rVert_{\fbl^p[E]}$.

  We claim that $f\notin I[E]$. By Lemma~\ref{L5.1}\ref{L5.1_0}, this amounts to showing that for every $m\in\N$ and $x_1,\ldots,x_m\in E$, we can find $x^*\in E^*$ such that $f(x^*)>\sum_{j=1}^m\lvert x^*(x_j)\rvert$. Since $m+1>m$, the  system of~$m$ homogeneous linear equations
  \begin{equation*}
  \sum_{k=1}^{m+1}b_k^*(x_j)t_k = 0\qquad (j=1,\ldots,m) \end{equation*}
  in the $m+1$ real variables $t_1,\ldots,t_{m+1}$ has a nontrivial solution, which we denote $(t_1,\ldots,t_{m+1})\in\R^{m+1}\setminus\{(0,\ldots,0)\}$.  Then the functional $x^*= \sum_{k=1}^{m+1} t_kb_k^*\in E^*$ satisfies $x^*(x_j)=0$ for each $j=1,\ldots,m$, and therefore
  \[ f(x^*) =\sum_{n=1}^{m+1}\frac{\lvert t_n\rvert}{2^n}>0 = \sum_{j=1}^m\lvert x^*(x_j)\rvert. \]

We have thus shown that conditions \ref{L5.3a}--\ref{L5.3g} are equivalent. Combining condi\-tions~\ref{L5.3c} and~\ref{L5.3h}, we see that they imply condition~\ref{L5.3d}.

\ref{L5.3d}$\Rightarrow$\ref{L5.3e}. Suppose that $I_{w^*}[E]=H_{w^*}[E]$. Then $I_{w^*}[E]=H_{w^*}^p[E]$, which is closed by Lemma~\ref{L5.1}\ref{L5.1ii'}.  

\ref{L5.3e}$\Rightarrow$\ref{L5.3f}. This follows from the general fact that $\fbl^p[E]\subseteq\overline{I_{w^*}[E]}$.

Finally, we can complete the proof by observing that the implication \ref{L5.3f}$\Rightarrow$\ref{L5.3g} is trivial.
\end{proof}

As a consequence, we obtain the following simplified version of the chains of inclusions exhibited in~\eqref{E:inclusions1} and~\eqref{E:inclusions2} when the Banach space $E$ is finite-dimensional.

 \begin{cor}\label{Cor:3.5} Let $E$ be a Banach space of finite dimension at least~$2$, and take $1\le p<\infty$. Then 
 \[ \fbl^p[E] = I_{w^*}[E] = 
   H_{w^*}[E]\subsetneq I[E] = J_{w^*}^p[E] = M_{w^*,0}^p[E]=H^p[E]. 
\]
 \end{cor}
 
 \begin{proof}
 The equalities follow from Theorem~\ref{L5.3} and Lemma~\ref{L5.1}\ref{lemma:Mw*0ii}, while Proposition~\ref{prop:findimcase} implies that the inclusion in the middle is proper.
 \end{proof}

\section{$H_{w^*}^p$ vs.\ $\operatorname{FBL}^p$ for infinite-dimensional spaces}\label{S:4} 

\noindent
Theorem~\ref{L5.3} shows that when the Banach space~$E$ is infinite-dimensional, we should replace  $I[E]$ and  $I_{w^*}[E]$ with their closures $\overline{I[E]}$ and  $\overline{I_{w^*}[E]}$ in the  norm~$\lVert\,\cdot\,\rVert_{\fbl^p[E]}$, for \mbox{$1\le p<\infty$}, before we compare them with the other Banach lattices we study. In this section we shall focus on the inclusions
\[ \fbl^p[E]\subseteq \overline{I_{w^*}[E]}\subseteq H_{w^*}^p[E]. \]
We note that all three sets are closed sublattices of $H^p[E]$, and Theorem~\ref{L5.3} shows that  $\fbl^p[E] = H_{w^*}^p[E]$ when~$E$ is finite-dimensional. 
However, at this point we do not know whether one or both  of these inclusions  could be an equality for some infinite-dimensional Banach space~$E$, raising the following two questions:

\begin{question}\label{Q5.7}
For  which Banach spaces~$E$ is it true that $\fbl^p[E] = \overline{I_{w^*}[E]}$ for some/all $1\le p<\infty$?
\end{question}

\begin{question}\label{Q5.8}
    For which Banach spaces $E$ is it true that $\overline{I_{w^*}[E]} = H_{w^*}^p[E]$  for some/all $1\le p<\infty$?
\end{question}

We begin by addressing the first of these  questions. Our aim is to prove the following result.

\begin{thm}\label{thm:fbl_neq_clIw_p} Let~$E$ be a Banach  space which admits an
infinite-di\-men\-sional, separable quotient space, and let $1\le p<\infty$. Then $I_{w^*}[E]\nsubseteq\fbl^p[E]$, and hence \[ \fbl^p[E]\subsetneq   \overline{I_{w^*}[E]}. \] 
\end{thm}

We remark that no infinite-dimensional Banach space~$E$ which fails to admit an infinite-dimensional, separable quotient space is known. In view of this, we conjecture that the condition that $\fbl^p[E] = \overline{I_{w^*}[E]}$ for some (or all) $1\le p<\infty$ characterizes that the Banach space~$E$ is finite-dimensional.

The following lemma will play a key role in the proof of Theorem~\ref{thm:fbl_neq_clIw_p}.

  \begin{lem}\label{lemma_sep_quotient} Let $E$ be a Banach space. The following conditions are equivalent:
    \begin{enumerate}[label={\normalfont{(\alph*)}}]
    \item\label{lemma_sep_quotient1} 
    For every $(\lambda_n)\in\ell_2^+$, 
    $E^*$ contains a weakly $1$-summable 
    sequence $(x_n^*)$ such that $\lVert x_n^*\rVert = \lambda_n$ for each $n\in\N$ and the subspace $\bigcup_{m\in\N}\bigcap_{n=m}^\infty\ker x^*_n$ is dense in~$E$.
    \item\label{lemma_sep_quotient1.25}
    For every $(\lambda_n)\in c_0^+$, $E^*$ contains a weakly $2$-summable sequence $(x_n^*)$ such that $\lVert x_n^*\rVert = \lambda_n$ for each $n\in\N$ and the subspace $\bigcup_{m\in\N}\bigcap_{n=m}^\infty\ker x^*_n$ is dense in~$E$.
    \item\label{lemma_sep_quotient1.5}  $E^*$ contains a  sequence $(x_n^*)$ of linearly independent functionals such that the subspace $\bigcup_{m\in\N}\bigcap_{n=m}^\infty\ker x^*_n$ is dense in~$E$.
    \item\label{lemma_sep_quotient2}  $E^*$ contains a  sequence $(x_n^*)$ of non-zero functionals such that the subspace $\bigcup_{m\in\N}\bigcap_{n=m}^\infty\ker x^*_n$ is dense in~$E$.
    \item\label{lemma_sep_quotient3} $E$ admits an infinite-dimensional, separable quotient space.
    \end{enumerate}
  \end{lem}

  \begin{proof}  We begin by showing that conditions~\ref{lemma_sep_quotient1.5}, \ref{lemma_sep_quotient2} and~\ref{lemma_sep_quotient3} are equivalent. It is clear that \ref{lemma_sep_quotient1.5}$\Rightarrow$\ref{lemma_sep_quotient2}.

\ref{lemma_sep_quotient2}$\Rightarrow$\ref{lemma_sep_quotient3}.   Suppose that
    $(x_n^*)$ is a sequence in~$E^*\setminus\{0\}$  such that $\bigcup_{m\in\N} W_m$ is dense in~$E$, where $W_m:=\bigcap_{n=m}^\infty\ker x^*_n$. Clearly $W_m\subseteq W_{m+1}$ for each $m\in\N$, and either $W_m=W_{m+1}$ or $W_m$~has codimension one in~$W_{m+1}$. We cannot have $W_m=W_{m+1}$ for all but finitely many $m\in\N$ because if we did, we would have $\bigcup_{m\in\N} W_m = W_{m_0}$ for some $m_0\in\N$, so $E=W_{m_0}$ because~$\bigcup_{m\in\N} W_m$ is dense in~$E$ and~$W_{m_0}$ is closed. This would imply that $x_n^*=0$ for each $n\ge m_0$, which contradicts our assumption. Hence, by passing to a subsequence, we may suppose that $W_m\subsetneq W_{m+1}$ for each $m\in\N$. According to a theorem of Mujica  (see \cite[p.~199]{Hajek_Biorth}), this implies that~\ref{lemma_sep_quotient3} is satisfied.

    Alternatively, we can easily verify~\ref{lemma_sep_quotient3} directly, showing that the quotient space $E/W_1$ is separable and infinite-dimensional. To this end, take $w_m\in W_{m+1}\setminus W_m$ such that $W_{m+1} = W_m + \R w_m$ for each $m\in\N$. Then the subspace 
    \[ W_1+\operatorname{span}\{w_m : m\in\N\} \] 
    is dense in~$E$, and therefore the subspace $\operatorname{span}\{Qw_m : m\in\N\}$ is dense in~$E/W_1$, where $Q\colon E\to E/W_1$ denotes the quotient map; hence~$E/W_1$ is separable. Furthermore, $E/W_1$ is infinite-dimensional because the sequence $(Qw_m)_{m\in\N}$ is linearly independent. Indeed, assuming the contrary, we can write $Qw_{m+1} = \sum_{j=1}^ms_jQw_j$ for some $m\in\N$ and some $s_1,\ldots,s_m\in\R$. Consequently $w_{m+1}- \sum_{j=1}^ms_jw_j\in W_1$, so $w_{m+1}\in W_{m+1}$, which is a contradiction.

    \ref{lemma_sep_quotient3}$\Rightarrow$\ref{lemma_sep_quotient1.5}. Suppose that~$W$ is a closed subspace of~$E$ such that $E/W$ is infinite-dimensional and separable. Then $E/W$ has a Markushevich basis (see for instance \cite[Theorem~4.59]{fab-ultimo}), say $(Qx_m)_{m\in\N}$, where $x_m\in E$ and $Q\colon E\to E/W$ denotes the quotient map.  Set  $x_n^* = Q^*(y_n^*)\in E^*$ for each $n\in\N$, where $y_n^*\in (E/W)^*$ is the $n^{\text{th}}$ biorthogonal functional of  $(Qx_m)_{m\in\N}$, and observe that $(x_n^*)_{n\in\N}$ is linearly independent because~$Q^*$ is injective.
    We claim that $\bigcup_{m\in\N}\bigcap_{n=m}^\infty\ker x_n^*$ is dense in~$E$. Indeed, given $x\in E$ and $\epsilon>0$, we can find $m\in\N$, $s_1,\ldots,s_m\in\R$ and $w\in W$ such that $\lVert x-\sum_{j=1}^ms_jx_j-w\rVert<\epsilon$ because $\operatorname{span}\{Qx_n : n\in\N\}$ is dense in~$E/W$. Now the conclusion follows because
    \[ \biggl\langle \sum_{j=1}^ms_jx_j+w, x_n^*\biggr\rangle = \sum_{j=1}^ms_j\langle Qx_j, y_n^*\rangle = 0 \] for each $n>m$, so $\sum_{j=1}^ms_jx_j+w\in \bigcap_{n=m+1}^\infty\ker x_n^*$.

    Clearly \ref{lemma_sep_quotient1}$\Rightarrow$\ref{lemma_sep_quotient2} and \ref{lemma_sep_quotient1.25}$\Rightarrow$\ref{lemma_sep_quotient2}. We shall now complete the proof by showing that    \ref{lemma_sep_quotient1.5}$\Rightarrow$\ref{lemma_sep_quotient1} and \ref{lemma_sep_quotient1.5}$\Rightarrow$\ref{lemma_sep_quotient1.25}. This is the harder part of the proof, but it is actually quite simple, as the former consists of making appropriate adjustments to the proof of the Dvoretzky--Rogers Theorem given in \cite[Theorem~1.2]{DJT}, and the latter is a minor variation, using ideas from \cite[Theorem~2]{Grothendieck}. We include the details for the convenience of the reader.

\ref{lemma_sep_quotient1.5}$\Rightarrow$\ref{lemma_sep_quotient1}. Let $(\lambda_n)\in\ell_2^+$, 
and set $n_0=0$. As in the proof of \cite[Theorem~1.2]{DJT}, 
we choose integers $1\le n_1<n_2<\cdots$ such that 
\begin{equation}\label{Eq:DJTproof1}
\sum_{j=n_k+1}^\infty\lambda_j^2\le 2^{-2k}\qquad (k\in\N).
\end{equation}
Suppose that $(y_n^*)$ is a linearly independent sequence in~$E^*$ such that the subspace
$\bigcup_{m\in\N}\bigcap_{n=m}^\infty\ker y^*_n$ is dense in~$E$. 
For each $k\in\N$,  we can apply \cite[Lemma~1.3]{DJT} to the $2(n_k-n_{k-1})$-dimensional space $\operatorname{span} \{ y_j^* : 2n_{k-1}<j\le 2n_k\}$ to find functionals \[ z_{n_{k-1}+1}^*,z_{n_{k-1}+2}^*,\ldots,z_{n_k}^*\in\operatorname{span} \{ y_j^* : 2n_{k-1}<j\le 2n_k\}, \] each having norm between~$1/2$ and~$1$, such that
\begin{equation}\label{lemma_sep_quotient:eq2} \biggl\|\sum_{j=n_{k-1}+1}^{n_k}\alpha_jz_j^*\biggr\|\le \biggl(\sum_{j=n_{k-1}+1}^{n_k}\alpha_j^2\biggr)^{1/2}\qquad (\alpha_{n_{k-1}+1},\ldots,\alpha_{n_k}\in\R). \end{equation}
Define $x_j^* = \lambda_jz_j^*/\lVert z_j^*\rVert\in E^*$ for every $j\in\{n_{k-1}+1,\ldots,n_k\}$, so that $\lVert x_j^*\rVert = \lambda_j$, and observe that
\begin{equation}\label{lemma_sep_quotient:eq1} \bigcap_{i=2n_{k-1}+1}^{2n_k}\ker y_i^*\subseteq\ker z_j^*\subseteq\ker x_j^*\qquad (j\in\{n_{k-1}+1,\ldots,n_k\}). \end{equation}
The argument given in~\cite{DJT} shows that the sequence $(x_j^*)$ is weakly $1$-summable. 

Hence it only remains to show that $\bigcup_{m\in\N}\bigcap_{n=m}^\infty\ker x_n^*$ is dense in~$E$. Given $x\in E$ and $\epsilon>0$,  choose $m\in\N$ and $y\in\bigcap_{n=m}^\infty\ker y_n^*$ such that $\lVert x-y\rVert<\epsilon$. Then, choosing $\ell\in\N_0$ such that $2n_\ell+1\ge m$, we obtain
\[ y\in \bigcap_{i=2n_\ell+1}^\infty\ker y_i^* = \bigcap_{k=\ell}^\infty\biggl(\bigcap_{i=2n_k+1}^{2n_{k+1}}\ker y_i^*\biggr)\subseteq \bigcap_{k=\ell}^\infty\biggl(\bigcap_{j=n_k+1}^{n_{k+1}}\ker x_j^*\biggr) = \bigcap_{j=n_\ell+1}^\infty\ker x_j^*, \]
where the inclusion in the middle follows from~\eqref{lemma_sep_quotient:eq1}. 

\ref{lemma_sep_quotient1.5}$\Rightarrow$\ref{lemma_sep_quotient1.25}. As already stated, this proof is a minor variation of the argument we have just given to show that \ref{lemma_sep_quotient1.5}$\Rightarrow$\ref{lemma_sep_quotient1}, so we shall focus on explaining the differences. It suffices to consider  $(\lambda_n)\in c_0^+$ with norm at most~$1$. We can no longer ensure that~\eqref{Eq:DJTproof1} holds; instead, we choose integers $0=n_0< n_1<n_2<\cdots$ such that 
\begin{equation}\label{lemma_sep_quotient:eq3}
    \sup_{j > n_k}\lambda_j\leq 2^{-k}\qquad (k\in\N).
\end{equation}
Our assumption that $\lVert(\lambda_n)\rVert_\infty\le1$ implies that this estimate is also valid for $k=0$.
Now, with $(y_n^*)$ chosen as before,  we can find the sequence $(z_j^*)$ and define $(x_j^*)$ as above. Then the proof that $\bigcup_{m\in\N}\bigcap_{n=m}^\infty\ker x_n^*$ is dense in~$E$ carries over verbatim, so we only need to verify that $(x_j^*)$ is weakly $2$-summable. For every $(a_j)\in B_{\ell_2}$ and $m\in\N$, we can apply~\eqref{lemma_sep_quotient:eq2}, \eqref{lemma_sep_quotient:eq3} and the fact that $\lVert z_j^*\rVert\ge\frac12$ to obtain   
\begin{align*} \biggl\|\sum_{j=1}^{n_m} a_jx_j^*\biggr\|&\leq \sum_{k=1}^m\biggl\|\sum_{j=n_{k-1}+1}^{n_{k}}a_jx_j^*\biggr\| = \sum_{k=1}^m\biggl\|\sum_{j=n_{k-1}+1}^{n_{k}}\frac{a_j\lambda_j}{\lVert z_j^*\rVert}z_j^*\biggr\|\\ &\leq \sum_{k=1}^m\biggl(\sum_{j=n_{k-1}+1}^{n_{k}}\frac{a_j^2\lambda_j^2}{\lVert z_j^*\rVert^2}\biggr)^{\frac12}\le 
\sum_{k=1}^m 2^{2-k}\biggl(\sum_{j=n_{k-1}+1}^{n_{k}}a_j^2\biggr)^{\frac12}\leq4. \end{align*}
This shows that $(x_j^*)$ is weakly $2$-summable because $n_m\to\infty$ as $m\to\infty$.
\end{proof}

\begin{proof}[Proof of Theorem~{\normalfont{\ref{thm:fbl_neq_clIw_p}}}] This proof is a more subtle variant of the proof of the implication \ref{L5.3c}$\Rightarrow$\ref{L5.3a} in Theorem~\ref{L5.3}. 

Recall that we seek a function $f\in I_{w^*}[E]\setminus\fbl^p[E]$. It turns out to be convenient to decompose~$E$ as  $E=F\oplus\R x_1$ and then carry out the main technical part of the construction of~$f$ within the subspace~$F$. Consequently we choose $x_1\in E$ and $x_1^*\in E^*$ such that $\lVert x_1\rVert =  \lVert x_1^*\rVert = 1 = x_1^*(x_1)$, and we then define $F=\ker x_1^*$. 

Since $E$ admits an
  infinite-dimensional, separable quotient space, so does~$F$; that is, $F$ satisfies one, and hence all, of the equivalent conditions in Lemma~\ref{lemma_sep_quotient}. Using this, we shall prove that, for a suitably chosen real number $q>p$, $F^*$~contains a weakly $p$-summable sequence $(y_n^*)$ such that 
  \begin{equation}\label{E:eq2} \lVert y^*_n\rVert = n^{-\frac1q}\qquad (n\in\N)\qquad\text{and}\qquad \overline{\bigcup_{m\in\N}\bigcap_{n\ge m}\ker y^*_n} = F. \end{equation} 
We split the proof of this claim  in two cases:
\begin{itemize}
    \item If $p<2$, we choose $q\in(p,2)$ and observe that $(n^{-\frac1q})_{n\in\N}\in\ell_2^+$. Therefore Lemma~\ref{lemma_sep_quotient}\ref{lemma_sep_quotient1} implies that~$F^*$ contains a weakly $1$-summable sequence $(y_n^*)$ which satisfies~\eqref{E:eq2}. Since $p\ge1$, it follows that $(y_n^*)$ is $p$-summable. 
    \item Otherwise  $p\ge 2$, in which case we choose $q\in(p,\infty)$. Since \mbox{$(n^{-\frac1q})_{n\in\N}\in c_0^+$}, Lemma~\ref{lemma_sep_quotient}\ref{lemma_sep_quotient1.25} shows that~$F^*$ contains a weakly $2$-summable sequence $(y_n^*)$ which satisfies~\eqref{E:eq2}. The fact that $p\ge2$ ensures that $(y_n^*)$ is $p$\nobreakdash-summable. 
\end{itemize}

Let~$P$ be the projection of~$E$ onto~$F$ given by $Px = x- x_1^*(x)x_1$ for $x\in E$. 
For each $n\in\N$, set $x_{n+1}^* = P^*(y_n^*)\in E^*$, and take a unit vector $x_{n+1}\in F$ such that $y^*_n(x_{n+1})\ge \lVert y_n^*\rVert/2$.  Then we have:
  \begin{itemize}
  \item $x_{n+1}^*(x_1) =0$ and $x_{n+1}^*(x_{n+1})\ge \lVert y_{n}^*\rVert/2$ for every $n\in\N$;
  \item the subspace $W=\bigcup_{m\in\N}W_m$ is dense in~$E$, where $W_m = \bigcap_{n=m}^\infty\ker x^*_n$;
  \item the sequence $(x_n^*)$ is weakly $p$-summable, so 
  \begin{equation}\label{E:eq1}
        K := \sup_{x\in B_E}\biggl(\sum_{n=1}^\infty \lvert x_n^*(x)\rvert^p\biggr)^{\frac1p} <\infty. 
  \end{equation}
  \end{itemize}
  Set $s_n = n^{\frac1q-\frac1p}\in(0,1]$ for each $n\in\N$. Then 
  $s_n\to 0$ as $n\to\infty$, so Lemma~\ref{Lemma:supinHw*} shows that the function $g = \bigvee_{n=1}^\infty \lvert \delta_{s_nx_{n+1}}\rvert$ belongs to~$H_{w^*}[E]$, and therefore $f =\lvert\delta_{x_1}\rvert\wedge g$ belongs to~$I_{w^*}[E]$ by  Lemma~\ref{L5.1}\ref{L5.1ii}.
We shall now complete the proof  by showing  that $f\notin \fbl^p[E]$, using $(x_n^*)$ to produce a suitable ``witness sequence''.

The sublattice~$L$ of~$H[E]$ generated by $\{\delta_w : w\in W\}$ is dense in $\fbl^p[E]$ because~$W$ is dense in~$E$ (see the construction of $\fbl^p[E]$ in \cite[Theorem 6.1]{JLTTT}).  
Therefore it will suffice to show that 
\begin{equation}\label{thm:fbl_neq_clIw_p:eq0} \inf\bigl\{\lVert f - h\rVert_{\fbl^p[E]} : h\in L\bigr\}>0. \end{equation} 
Since~$W$ is a subspace of~$E$, every function $h\in W$ can be expressed in the form $h = \bigvee_{i=1}^k \delta_{v_i} - \bigvee_{i=1}^k \delta_{w_i}$ for some $k\in\N$ and some $v_1,\ldots,v_k,w_1,\ldots,w_k\in W$. As~$W$ is the union of the increasing sequence $(W_m)$ of subspaces, we can find $m\in\N$ such that $v_1,\ldots,v_k,w_1,\ldots,w_k\in W_m$. Then, for each $j\in\{1,\ldots,m\}$, we see that the functional $z_j^* = m^{-\frac1p} x_1^* + x_{m+j}^*\in E^*$ satisfies
\begin{align}\label{eq:0607eq3new}  h(z_j^*) &=  \bigvee_{i=1}^k \bigl(m^{-\frac1p} x_1^*(v_i) + x_{m+j}^*(v_i)\bigr) - \bigvee_{i=1}^k \bigl(m^{-\frac1p}x_1^*(w_i) + x_{m+j}^*(w_i)\bigr)\notag\\  &=   \bigvee_{i=1}^k m^{-\frac1p} x_1^*(v_i) - \bigvee_{i=1}^k m^{-\frac1p}x_1^*(w_i) = m^{-\frac1p}h(x_1^*) \end{align}
because $v_1,\ldots,v_k,w_1,\ldots,w_k\in W_m\subseteq\ker x_{n}^*$ for every $n\ge m$. 

It follows from~\eqref{E:eq1} that $\lVert(x_j^*)_{j=m+1}^{2m}\rVert_{p,{\normalfont{\text{weak}}}}\le K$. Combining this estimate with the easy observation that 
\begin{equation}\label{E:lowerestimate0} \Bigl\lVert\Bigl(\underbrace{m^{-\frac1p} x_1^*,\ldots,m^{-\frac1p} x_1^*}_{m}\Bigr)\Bigr\rVert_{p,{\normalfont{\text{weak}}}} = \lVert x^*_1\rVert = 1 \end{equation}
and the subadditivity of~$\lVert\,\cdot\,\rVert_{p,{\normalfont{\text{weak}}}}$, we see that $\lVert(z_j^*)_{j=1}^m\rVert_{p,{\normalfont{\text{weak}}}}\le K+1$, and consequently
\begin{equation}\label{E:lowerestimate2} (K+1)\lVert f-h\rVert_{\fbl^p[E]}\ge \biggl(\sum_{j=1}^{m}\bigl\lvert f(z_j^*)-h(z_j^*)\bigr\rvert^p\biggr)^\frac1p.  \end{equation}
Another application of~\eqref{E:lowerestimate0} shows that
\begin{equation}\label{E:lowerestimate1} \lVert f-h\rVert_{\fbl^p[E]}\ge \biggl(\sum_{j=1}^{m}\bigl\lvert m^{-\frac1p}\bigl(f(x_1^*)-h(x_1^*)\bigr)\bigr\rvert^p\biggr)^{\frac1p}.
\end{equation}
Now we combine~\eqref{E:lowerestimate2} and~\eqref{E:lowerestimate1} with the subadditivity of the norm on~$\ell_p^{m}$ to obtain
\begin{align}\label{E:lowerestimate3}  (K+2)\lVert f-h\rVert_{\fbl^p[E]}&\ge \biggl(\sum_{j=1}^{m}\bigl\lvert f(z_j^*)-h(z_j^*) - m^{-\frac1p}\bigl(f(x_1^*) - h(x_1^*)\bigr)\bigr\rvert^p\biggr)^{\frac1p}\notag\\ &= \biggl(\sum_{j=1}^{m}\bigl\lvert f(z_j^*) - m^{-\frac1p}f(x_1^*)\bigr\rvert^p\biggr)^\frac1p,  \end{align}
where the final equality follows from~\eqref{eq:0607eq3new}.

To find a lower bound on the right-hand side of~\eqref{E:lowerestimate3}, we recall that $f =\lvert\delta_{x_1}\rvert\wedge g$, where $g(x_1^*) =0$  because $x_{n+1}\in F= \ker x_1^*$ for every $n\in\N$, so $f(x_1^*) = 0$, while
\begin{align*} f(z_j^*) &= \bigl|m^{-\frac1p} x_1^*(x_1) + x_{m+j}^*(x_1)\bigr|\wedge \bigvee_{n=1}^\infty s_n\bigl| m^{-\frac1p}x_1^*(x_{n+1}) + x_{m+j}^*(x_{n+1})\bigr|\\ &= m^{-\frac1p}\wedge \bigvee_{n=1}^\infty s_n\bigl\lvert x_{m+j}^*(x_{n+1})\bigr\rvert\ge  m^{-\frac1p}\wedge  s_{m+j-1}\bigl\lvert x_{m+j}^*(x_{m+j})\bigr\rvert\\ &\ge m^{-\frac1p}\wedge \frac{s_{m+j-1}\lVert y_{m+j-1}^*\rVert}2 = m^{-\frac1p}\wedge \frac12 (m+j-1)^{-\frac1p} 
\ge \frac{1}{2}(2m)^{-\frac1p}
\end{align*} for each $j\in\{1,\ldots,m\}$. Substituting these estimates into~\eqref{E:lowerestimate3}, we conclude that
\begin{align*} (K+2)\lVert f-h\rVert_{\fbl^p[E]} &\ge
2^{-\frac{p+1}{p}}, 
\end{align*} 
from which~\eqref{thm:fbl_neq_clIw_p:eq0} follows because the constant~$K$ is  independent of~$h$.

This completes the proof that $f\in I_{w^*}[E]\setminus\fbl^p[E]$. The final clause follows because $\fbl^p[E]\subseteq\overline{I_{w^*}[E]}$. 
\end{proof}

Moving on to Question~\ref{Q5.8}, we shall only give one result which shows that this inclusion need not be proper, and finding a general answer may not be easy.

\begin{prop}
For $1\leq p<\infty$ and any set $A$, 
\[ \overline{I_{w^*}[\ell_1(A)]} = H_{w^*}^p[\ell_1(A)], \] 
where the closure is taken with respect to the norm $\|\cdot\|_{\fbl^p[\ell_1(A)]}$.
\end{prop}

\begin{proof} Set $E = \ell_1(A)$. 
Let us start by noting that \cite[Theorem 2.8]{OTTT} shows that $\operatorname{lat}\delta[E]$ is order dense in $H_{w^*}[E]$. Hence, for every $f\in H_{w^*}^p[E]_+$, the set 
\[ S_f=\{g\in \operatorname{lat}\delta[E]:0\leq g\leq f\} \]
satisfies  $f=\sup S_f$. Furthermore, $S_f$ is upwards directed because $g_1\vee g_2\in S_f$ whenever $g_1,g_2\in S_f$, and we have $\sup_{g\in S_f}\|g\|_{\fbl^p[E]}\leq\|f\|_{\fbl^p[E]}$.

Consequently, we can use the fact that $\fbl^p[E]$ has the strong Nakano property (this is proved for $p=1$ in \cite[Theorem 4.11]{ART}, with a similar argument giving the result for $p>1$, see \cite[Section 4]{ABT} for details) to conclude that there exists $h\in \fbl^p[E]\subset \overline{I_{w^*}[E]}$ such that $h\geq g$ for every $g\in S_f$. Then $h\ge\sup S_f = f$, so we get that $f\in\overline{I_{w^*}[E]}$. Hence, $H_{w^*}^p[E]\subseteq \overline{I_{w^*}[E]}$. The opposite inclusion follows from Lemma~\ref{L5.1}\ref{lemma:Mw*0iii}.
\end{proof}

\section{Characterizations of reflexivity}\label{S:5}
\noindent 
To understand the case where the underlying Banach space $E$ is reflexive, we begin with the observation that $E^{**}\subseteq H[E]$ because every functional $x^{**}\in E^{**}$ is positively homogeneous. Under this inclusion, the canonical image of $E$ inside $E^{**}$ is identified with $\{ \delta_x : x\in E\}$. The following lemma clarifies where the elements of~$E^{**}$ sit in the hierarchy of sublattices \eqref{E:inclusions1}--\eqref{E:inclusions2} of~$H[E]$ that we study. 

\begin{lem}\label{L5.4}
  Let $E$ be a Banach space, and take $1\le p<\infty$. Then:
\begin{enumerate}[label={\normalfont{(\roman*)}}]
\item\label{L5.1iv} $E^{**}\subseteq H^p[E]$ with $\lVert x^{**}\rVert_{\fbl^p[E]} =\lVert x^{**}\rVert$ for every $x^{**}\in E^{**}$.
\item\label{L5.1v} $E^{**}\cap I[E]=E^{**}\cap H_{w^*}[E] = E^{**}\cap M_{w^*,0}^p[E]=\{ \delta_x : x\in E\}$.
\end{enumerate}
\end{lem}

\begin{proof}
\ref{L5.1iv}.  The Hahn--Banach Theorem implies that
  \[ \biggl(\sum_{j=1}^m \lvert x^{**}(x_j^*)\rvert^p\biggr)^{\frac1p}\le \lVert x^{**}\rVert\cdot \lVert(x_j^*)_{j=1}^m\rVert_{p,{\normalfont{\text{weak}}}} \]
  for $x^{**}\in E^{**}$,  $m\in\N$ and $x_1^*,\ldots,x_m^*\in E^*$, and hence $\lVert x^{**}\rVert_{\fbl^p[E]}\le\lVert x^{**}\rVert$. The opposite inequality follows from the fact that
  $\lVert x^*\rVert_{p,{\normalfont{\text{weak}}}}=\lVert x^*\rVert$ for every $x^*\in E^*$.

  \ref{L5.1v}. It is clear that $\delta_x\in E^{**}\cap I[E]\cap H_{w^*}[E]$ for every $x\in E$, and Lemma~\ref{L5.1}\ref{lemma:Mw*0i} and~\ref{lemma:Mw*0ii} show that $I[E]\subseteq M_{w^*,0}^p[E]$. Hence we can complete the proof by showing that if $x^{**}\in E^{**}$ belongs to~$H_{w^*}[E]$ or~$M_{w^*,0}^p[E]$, then $x^{**}=\delta_x$ for some \mbox{$x\in E$}. 
  In both cases $x^{**}{\upharpoonright_{B_{E^*}}}$ is weak*-continuous at~$0$, so
  \mbox{$x^{**}{\upharpoonright_{B_{E^*}}^{\,-1}}(\{0\}) = \ker x^{**}\cap B_{E^*}$} is weak*-closed. This implies that 
  $x^{**}=\delta_x$ for some $x\in E$  
  by an application of the Krein--Smulian Theorem (see, \emph{e.g.,} \cite[Corollary~3.94]{fab-ultimo}).  
\end{proof}

\begin{cor}\label{C5.5}
  Let $E$ be a non-reflexive Banach space and $1\le p<\infty$. Then  \[ H^p[E]\nsubseteq M_{w^*,0}^p[E]\cup H_{w^*}[E]. \]
\end{cor}

\begin{proof}
  Lemma~\ref{L5.4} shows that $E^{**}\setminus\{\delta_x:x\in E\}\subseteq H^p[E]\setminus (M_{w^*,0}^p[E]\cup H_{w^*}[E])$.
\end{proof}

Before stating our next result, we recall the following classical characterization of weak se\-quen\-tial convergence in $C(K)$-spaces (see \cite[Theorem~1.1]{DS} or \cite[Corollary~IV.6.4]{NDJTS}), which will be required in the proof.
\begin{thm}\label{DomConvThm}
Let $K$ be a compact Hausdorff space. A  sequence $(f_n)_{n\in\N}$ in~$C(K)$ converges weakly to a function $f\in C(K)$ if and only if $\sup_{n\in\N}\lVert f_n\rVert_\infty<\infty$ and $f_n(x)\to f(x)$ as $n\to\infty$ for every $x\in K$.
\end{thm}

\begin{lem}\label{lem:reflexive}
Let $1\le p<\infty$, let~$E$ be a reflexive Banach space, and equip its unit ball~$B_E$ with the weak topology. Then the  map $\psi_p\colon B_{E^*}\rightarrow C(B_E)$ defined by 
\begin{equation*}
\psi_p(x^*)(x) = \lvert x^*(x)\rvert^p\qquad (x^*\in B_{E^*},\, x\in B_E) \end{equation*} 
is continuous with respect to the weak topologies on $B_{E^*}$ and $C(B_E)$, respectively.
\end{lem}

\begin{proof}
We begin by showing that $\psi_p$ is weakly sequentially continuous. Take a sequence $(x_n^*)_{n\in\mathbb N}$ in~$B_{E^*}$ which converges weakly to some~$x^*\in B_{E^*}$. Then $\sup_{n\in\N}\lVert\psi_p(x_n^*)\rVert_\infty\le 1$ and \[ \psi_p(x^*_n)(x)= \lvert x_n^*(x)\rvert^p\underset{n\to\infty}\longrightarrow \lvert x^*(x)\rvert^p = \psi_p(x^*)\qquad (x\in B_E), \]
so Theorem~\ref{DomConvThm} shows that $(\psi_p(x_n^*))_{n\in\N}$ converges weakly to~$\psi_p(x^*)$. 

Now, in order to complete the proof, take a weakly closed subset $C$ of $C(B_E)$, and let us prove that $\psi_p^{-1}(C)$ is weakly closed in $B_{E^*}$. By the reflexivity of $E$ and the Eber\-lein--Smulian Theorem, this is equivalent to showing that $\psi_p^{-1}(C)$ is weakly sequentially compact. 
To verify this, take a sequence $(x_n^*)_{n\in\mathbb N}$ in $\psi_p^{-1}(C)$. Due to the weak sequential compactness of $B_{E^*}$, $(x_n^*)_{n\in\N}$ has a subsequence $(x_{n_j}^*)_{j\in\N}$ which converges weakly to some $x^*\in B_{E^*}$. The first part of the proof implies that $(\psi_p(x_{n_j}^*))_{j\in\N}$ converges weakly to $\psi_p(x^*)$. Since $\psi_p(x_{n_j}^*)$ belongs to the weakly closed set~$C$ for each $j\in\N$, we see that $\psi_p(x^*)\in C$. In conclusion, we have shown that $(x_n^*)_{n\in\mathbb N}$ has a subsequence which converges weakly to $x^*\in\psi_p^{-1}(C)$, and the result follows.
\end{proof}

Let~$E$ be a Banach space, equip~$B_{E^{**}}$ with the relative weak* topology, and take $\mu\in C(B_{E^{**}})^*_+$ (so in other words~$\mu$ is a positive Radon measure on~$B_{E^{**}}$). Then, for every $1\le p<\infty$, we can define a function $f_\mu^p\colon E^*\to\R$ by
\begin{equation}\label{eq:defnfmu} f_\mu^p(x^*)=\biggl(\int_{B_{E^{**}}} |x^{**}(x^*)|^p\, d\mu(x^{**})\biggr)^{\frac1p}\qquad (x^*\in E^*). \end{equation}
This construction is useful for our purposes because $f_\mu^p \in H^p[E]$ with $\lVert f_\mu^p\rVert_{\fbl^p[E]}\le \lVert \mu\rVert$; this was shown in \cite[Proposition~7.6]{JLTTT} in the case where $\lVert\mu\rVert=1$; the general case follows by scaling.

\begin{thm}\label{t:ref}
Let $1\le p<\infty$, and let~$E$ be a Banach space. The following con\-di\-tions are equivalent:
\begin{enumerate}[label={\normalfont{(\alph*)}}]
    \item\label{ref:a} $E$ is reflexive.
    \item\label{ref:a'} $f_\mu^p\in H_{w^*}[E]$ for every $\mu\in C(B_{E^{**}})^*_+$.
    \item\label{ref:b} $H^p[E] = J_{w^*}^p[E]$.
    \item\label{ref:c} $H^p[E] = \overline{J_{w^*}^p[E]}$.
    \item\label{ref:d} $f_\mu^p\in M_{w^*,0}^p[E]$ for every $\mu\in C(B_{E^{**}})^*_+$.
\end{enumerate}
\end{thm}

\begin{proof}
\ref{ref:a}$\Rightarrow$\ref{ref:a'} Suppose that $E$ is reflexive, so that $E^{**}=E$ with the weak* and weak topologies coinciding, and take $\mu\in C(B_E)^*_+$. For $x^*\in B_{E^*}$, we have 
\[ f_\mu^p(x^*)^p = \int_{B_{E}} |x^{*}(x)|^p\, d\mu(x) =\langle \psi_p(x^*),\mu\rangle = (\mu\circ\psi_p)(x^*), \]
so that $(f_\mu^p{\upharpoonright_{B_{E^*}}})^p=\mu\circ\psi_p$, where $\psi_p$ is the map defined in Lemma~\ref{lem:reflexive}. Take a net $(x^*_\alpha)$ in $B_{E^*}$ which weak*-converges to~$x^*$. Then, by reflexivity, $(x^*_\alpha)$ converges weak\-ly to~$x^*$, so Lemma~\ref{lem:reflexive} implies that $(\psi_p(x^*_\alpha))$ converges weakly to $\psi_p(x^*)$, which means that $(\langle \psi_p(x_\alpha^*),\mu\rangle)=(f_\mu^p(x^*_\alpha)^p)$ converges to $\langle \psi_p(x^*),\mu\rangle = f_\mu^p(x^*)^p$. This shows that $f_\mu^p\in H_{w^*}[E]$. 

\ref{ref:a'}$\Rightarrow$\ref{ref:b}. Suppose that $f_\mu^p\in H_{w^*}[E]$ for every $\mu\in C(B_{E^{**}})^*_+$, and recall that $f_\mu^p$ always belongs to~$H^p[E]$. For every $f\in H^p[E]$, \cite[Proposition~7.7]{JLTTT} shows that we can find a probability measure $\mu\in C(B_{E^{**}})^*_+$ such that
\[ |f(x^*)|\leq \lVert f\rVert_{\fbl^p[E]}\,f_\mu^p(x^*)\qquad (x^*\in E^*), \] and therefore $f\in J_{w^*}^p[E]$.

\ref{ref:b}$\Rightarrow$\ref{ref:c}  is trivial.

\ref{ref:c}$\Rightarrow$\ref{ref:d}. This follows from Lemma~\ref{L5.1}\ref{lemma:Mw*0ii} because  
$f_\mu^p\in H^p[E]$ for every $\mu\in C(B_{E^{**}})^*_+$.

Finally, we prove the implication \ref{ref:d}$\Rightarrow$\ref{ref:a} by contraposition. Suppose that $E$ is non-reflexive. Then $B_{E^*}$ contains a net $(x^*_\alpha)$ which weak*-converges to~0, but does not converge weakly to~0. 
Take $x_0^{**}\in E^{**}$ such that $(x_0^{**}(x^*_\alpha))$ does not converge to~0, 
and consider the Dirac measure $\mu= \delta_{x^{**}_0}\in C(B_{E^{**}})_+^*$; it satisfies 
\begin{align*}
    f_\mu^p(x_\alpha^*) &= \biggl(\int_{B_{E^{**}}} |x^{**}(x^*_\alpha)|^p\,d\mu(x^{**})\biggr)^{\frac1p} = \bigl|x^{**}_0(x_\alpha^*)\bigr|\not\rightarrow 0 = f_\mu^p(0).
\end{align*}
This shows that  $f_{\mu}^p{\upharpoonright_{B_{E^*}}}$ fails to be weak*-continuous at $0$, so $f_\mu^p\notin M_{w^*,0}^p[E]$.
\end{proof}

\section{Lattice homomorphisms}\label{section:isomorphism}
\noindent
It is well known that for compact Hausdorff spaces~$K$ and~$L$, every lattice homo\-mor\-phism  $T\colon C(K)\rightarrow C(L)$ arises by combining a composition operator with point\-wise multi\-pli\-ca\-tion by a positive continuous function. (A precise statement of this result can for instance be found in \cite[Theorem~3.2.10]{MN}.) The aim of this section is to obtain an analogous representation for lattice homomorphisms between free Banach lattices and related lattices of positively homogeneous functions, and then study the properties and consequences of this representation.

We require a new quantity. Given $1\leq p<\infty$, Banach spaces $E$ and $F$, and a positively homogeneous map $\Phi\colon F^*\to E^*$, set
\begin{equation}\label{defn:Phi_p_norm}  \lVert \Phi\rVert_p=\sup\left\{\bigl\lVert \bigl(\Phi (y_j^*)\bigr)_{j=1}^m\bigr\rVert_{p,{\normalfont{\text{weak}}}} : m\in\mathbb N,\, (y_j^*)_{j=1}^m\subset F^*,\,\lVert (y_j^*)_{j=1}^m\rVert_{p,{\normalfont{\text{weak}}}}\leq1\right\}. 
\end{equation}
This definition can be viewed as a vector-valued version of the $\fbl^p$-norm~\eqref{Defn:FBLpNorm} in the precise sense that $\lVert \Phi\rVert_p=\lVert \Phi\rVert_{\fbl^p[F]}$ when $E=\R$. 
In the general case (where~$E$ need not be $1$-dimensional), it is clear that $\|\Phi\|_p<\infty$ precisely when~$\Phi$ maps weakly $p$-summable sequences in~$F^*$ to weakly $p$-summable sequences in~$E^*$. 

We begin by establishing some basic properties of the operator given by composition with a positively homogeneous map.

\begin{lem}\label{L:compositionOps} Let $\Phi\colon F^*\to E^*$ be a positively homogeneous map for some Banach spaces~$E$ and~$F$. 
\begin{enumerate}[label={\normalfont{(\roman*)}}]
\item\label{L:compositionOps:i} The composition operator 
\[ C_\Phi\colon f\mapsto f\circ\Phi \]
defines a lattice homomorphism $C_\Phi\colon H[E]\to H[F]$. 
\item\label{L:compositionOps:ii} Suppose that $\|\Phi\|_p<\infty$ for some $1\le p<\infty$.
Then $\Phi(B_{F^*})\subseteq \|\Phi\|_p B_{E^*}$, 
\begin{equation}\label{L:compositionOps:eq1} C_{\Phi}(H^p[E])\subseteq H^p[F], \end{equation}
and $C_\Phi$ is bounded with norm $\|C_{\Phi}\|=\|\Phi\|_p$ when viewed as a linear opera\-tor $(H^p[E],\|\cdot\|_{\fbl^p[E]})\to(H^p[F],\|\cdot\|_{\fbl^p[F]})$.

Suppose in addition that $\Phi{\upharpoonright_{B_{F^*}}}$ is weak*-to-weak* continuous. Then
\begin{equation}\label{L:compositionOps:eq2}
\begin{alignedat}{2} C_\Phi(H_{w^*}[E])&\subseteq H_{w^*}[F],\qquad & C_{\Phi}(H^p_{w^*}[E])&\subseteq H^p_{w^*}[F],\\  
C_{\Phi}(J^p_{w^*}[E])&\subseteq J^p_{w^*}[F],\qquad & C_{\Phi}(\overline{J^p_{w^*}[E]})&\subseteq \overline{J^p_{w^*}[F]},\\
C_{\Phi}(M^p_{w^*,0}[E])&\subseteq M^p_{w^*,0}[F].
\end{alignedat}
\end{equation}
\item\label{L:compositionOps:iii} The map $p\mapsto\lVert\Phi\rVert_p$ is decreasing; that is, $\lVert\Phi\rVert_q\le\lVert\Phi\rVert_p$ for $1\le p<q<\infty$.
\end{enumerate}
\end{lem}

\begin{proof} 
\ref{L:compositionOps:i}. The composite map $f\circ\Phi$ is clearly positively homogeneous when~$f$ and~$\Phi$ are, so $C_\Phi(H[E])\subseteq H[F]$, and $C_\Phi$ is a lattice homomorphism because the vector lattice operations are defined coordinatewise.

\ref{L:compositionOps:ii}. Suppose that $\lVert\Phi\rVert_p<\infty$. Then~$\lVert \Phi(y^*)\rVert\le \|\Phi\|_p$ for every $y^*\in B_{F^*}$ because $\lVert z^*\rVert_{p,{\normalfont{\text{weak}}}} = \lVert z^*\rVert$ for every functional~$z^*$ (in either~$E^*$ or~$F^*)$, and there\-fore $\Phi(B_{F^*})\subseteq \|\Phi\|_p B_{E^*}$. Furthermore, working straight from the definitions~\eqref{Defn:FBLpNorm} and~\eqref{defn:Phi_p_norm}, we obtain
\begin{multline*}
\lVert C_\Phi (f)\rVert_{\fbl^p[F]}=\sup\biggl\{\Bigl(\sum_{j=1}^m\lvert f(\Phi y_j^*)\rvert^p\Bigr)^{\frac1p} : m\in\mathbb N,\, (y_j^*)_{j=1}^m\subset F^*,\biggr.\\[-3mm] \biggl.\lVert (y_j^*)_{j=1}^m\rVert_{p,{\normalfont{\text{weak}}}}\leq1\biggr\}\leq\|\Phi\|_p\, \|f\|_{\fbl^p[E]}  
\end{multline*}
for every $f\in H^p[E]$. This shows that~$C_{\Phi}$ maps~$H^p[E]$ into~$H^p[F]$ and $\|C_\Phi\|\leq \|\Phi\|_p$. On the other hand, given $(y_j^*)_{j=1}^m\subset F^*$ for some $m\in\N$ and $x\in B_E$, we have
\begin{align*}
\Bigl(\sum_{j=1}^m |(\Phi y_j^*)(x)|^p\Bigr)^{\frac{1}{p}}&=\Bigl(\sum_{j=1}^m |\delta_x(\Phi y_j^*)|^p\Bigr)^{\frac{1}{p}}=\Bigl(\sum_{j=1}^m |(C_\Phi \delta_x) (y_j^*)|^p\Bigr)^{\frac{1}{p}}\\
&\leq\|C_{\Phi} \delta_x\|_{\fbl^p[F]}\, \lVert (y_j^*)_{j=1}^m\rVert_{p,{\normalfont{\text{weak}}}}\leq\|C_{\Phi}\|\, \lVert (y_j^*)_{j=1}^m\rVert_{p,{\normalfont{\text{weak}}}},
\end{align*}
which shows that $\|\Phi\|_p\leq \|C_{\Phi}\|$, and therefore $\|C_{\Phi}\|=\|\Phi\|_p$.

Now suppose that $\Phi{\upharpoonright_{B_{F^*}}}$ is weak*-to-weak* continuous (as well as $\lVert\Phi\rVert_p<\infty$). Since $\Phi(B_{F^*})\subseteq \|\Phi\|_p B_{E^*}$ and the functions in~$H_{w^*}[E]$ are positively homo\-ge\-neous, we see that~$C_\Phi$ maps~$H_{w^*}[E]$ into~$H_{w^*}[F]$ and~$M^p_{w^*,0}[E]$ into~$M^p_{w^*,0}[F]$. 
The inclusion $C_{\Phi}(H^p_{w^*}[E])\subseteq H^p_{w^*}[F]$ follows immediately from the inclusions $C_\Phi(H_{w^*}[E])\subseteq H_{w^*}[F]$ and $C_{\Phi}(H^p[E])\subseteq H^p[F]$.

Finally, for each $f\in J^p_{w^*}[E]$, we can find $g\in H^p_{w^*}[E]$ such that $|f|\leq g$. Then we have $|C_{\Phi}(f)|\leq C_{\Phi}(g)\in H^p_{w^*}[F]$ by the previous inclusion, and con\-se\-quent\-ly $C_{\Phi}(f)\in J^p_{w^*}[F]$. This implies that $C_{\Phi}(\overline{J^p_{w^*}[E]})\subseteq \overline{J^p_{w^*}[F]}$ because~$C_\Phi$ is con\-ti\-nuous with respect to the $\fbl^p$-norm, as we have shown above.

\ref{L:compositionOps:iii}. This proof is similar to the proof of \cite[Proposition~7.3]{JLTTT}. We may suppose that $\lVert\Phi\rVert_p<\infty$, as otherwise there is nothing to prove. Take $x\in B_E$ and $(y_j^*)_{j=1}^m\subset F^*$ with $\lVert(y_j^*)_{j=1}^m\rVert_{q,\text{weak}}\le 1$ for some $m\in\N$, and set $\lambda_j = \lvert\langle x,\Phi(y_j^*)\rangle\rvert^{q/r}\ge 0$ for $1\le j\le m$, where $r = pq/(q-p)\in(p,\infty)$. Then $(1+r/q)p=r$, which together with the positive homogeneity of~$\Phi$ implies that
\begin{align}\label{eq:ptoq}
\sum_{j=1}^m\lambda_j^r &= \sum_{j=1}^m\bigl(\lambda_j\lvert\langle x,\Phi(y_j^*)\rangle\rvert\bigr)^p = \sum_{j=1}^m\lvert\langle x,\Phi(\lambda_jy_j^*)\rangle\rvert^p\notag\\ 
&\le \bigl\lVert\bigl(\Phi(\lambda_jy_j^*)\bigr)_{j=1}^m\bigr\rVert_{p,\text{weak}}^p\le \lVert\Phi\rVert_p^p\, \lVert(\lambda_jy_j^*)_{j=1}^m\rVert_{p,\text{weak}}^p.
\end{align}
The choice of~$r$ means that the conjugate exponent of~$r/p\in(1,\infty)$ is~$q/p$. Hence, for
$y\in B_F$, H\"{o}lder's inequality gives
\begin{align*}
\sum_{j=1}^m \lvert\langle y, \lambda_jy_j^*\rangle\rvert^p &= \sum_{j=1}^m \lambda_j^p\lvert\langle y,y_j^*\rangle\rvert^p\le \Bigl(\sum_{j=1}^m \lambda_j^r\Bigr)^{\frac{p}{r}} \Bigl(\sum_{j=1}^m\lvert\langle y,y_j^*\rangle\rvert^q\Bigr)^{\frac{p}{q}}\leq \Bigl(\sum_{j=1}^m\lambda_j^r\Bigr)^{\frac{p}{r}}.
\end{align*}
This shows that $\lVert(\lambda_jy_j^*)_{j=1}^m\rVert_{p,\text{weak}}^p\le \bigl(\sum_{j=1}^m\lambda_j^r\bigr)^{p/r}$, so by~\eqref{eq:ptoq}, we have \[ \sum_{j=1}^m\lambda_j^r\le \lVert\Phi\rVert_p^p\,\Bigl(\sum_{j=1}^m\lambda_j^r\Bigr)^{\frac{p}{r}}. \]  Rearranging this inequality and using once more that $1-p/r=p/q$, we obtain
\[ \lVert\Phi\rVert_p^p\ge \Bigl(\sum_{j=1}^m\lambda_j^r\Bigr)^{\frac{p}{q}} = \Bigl(\sum_{j=1}^m\lvert\langle x,\Phi(y_j^*)\rangle\rvert^q\Bigr)^{\frac{p}{q}}. \]
Now the conclusion follows by taking the $p^{\text{th}}$ root and the supremum over~$x$ and $(y_j^*)_{j=1}^m$.
\end{proof}

\begin{example}\label{Ex:FBLnotPhiInvariant} 
Identify~$\ell_\infty$ with~$\ell_1^*$, as usual. By \cite[Remark 10.6]{OTTT}, there is a positively ho\-mo\-ge\-neous map $\Phi\colon\ell_\infty\rightarrow \ell_\infty$ satisfying the hypotheses of Lemma~\ref{L:compositionOps}\ref{L:compositionOps:ii} for $p=1$, that is, \mbox{$\lVert\Phi\rVert_1<\infty$}  and $\Phi{\upharpoonright_{B_{\ell_\infty}}}$ is weak*-to-weak* continuous, but the corresponding composition operator~$C_\Phi$ does not map~$\fbl^1[\ell_1]$ into itself. 

We can now provide a much more general version of this result. Take $1\le p<\infty$, and let~$E$ and~$F$ be Banach spaces, where~$E$ is non-zero and~$F$ admits an in\-finite-di\-men\-sional, separable quotient space. By Theorem~\ref{thm:fbl_neq_clIw_p}, we can find a function \mbox{$f\in H_{w^*}^p[F]\setminus\fbl^p[F]$}. (In fact, we can take $f\in I_{w^*}[F]\setminus\fbl^p[F]$, but this will not help the following argument.) Choose $x_0^*\in E^*\setminus\{0\}$, and define a map $\Phi\colon F^*\rightarrow E^*$ by $\Phi(y^*)=f(y^*)x_0^*$ for $y^*\in F^*$. Since $f\in H_{w^*}^p[F]$, it is straightforward to check that~$\Phi$ is positively homo\-ge\-neous, $\Phi{\upharpoonright_{B_{F^*}}}$ is weak*-to-weak* continuous, and $\lVert\Phi\rVert_p = \lVert f\rVert_{\fbl^p[F]}\,\lVert x^*_0\rVert<\infty$.
However, if we pick $x_0\in E$ such that $x_0^*(x_0)=1$, then 
\[ C_\Phi(\delta_{x_0})(y^*)=(\delta_{x_0}\circ\Phi)(y^*) = \Phi(y^*)(x_0) =f(y^*)\qquad (y^*\in F^*), \]
so $C_\Phi(\delta_{x_0})=f\notin\fbl^p[F]$. Hence~$C_\Phi$ does not map~$\fbl^p[E]$ into~$\fbl^p[F]$.
\end{example}

The following question arises naturally from the inclusions~\eqref{L:compositionOps:eq1}--\eqref{L:compositionOps:eq2} and Example~\ref{Ex:FBLnotPhiInvariant}. We have not been able to answer it, partly due to the lack of progress on Question~\ref{Q5.8}.

\begin{question}\label{Q:Iw*Invariant} Let $E$ and $F$ be infinite-dimensional Banach spaces, and suppose that $\Phi\colon F^*\to E^*$ is a positively homogeneous map which satisfies $\lVert\Phi\rVert_p<\infty$ for some $1\le p<\infty$ and $\Phi{\upharpoonright_{B_{F^*}}}$ is weak*-to-weak* continuous. Is \[ C_\Phi(\overline{I_{w^*}[E]})\subseteq\overline{I_{w^*}[F]}\,? \] 
\end{question}

We shall now present our main representation theorem for lattice homomorphisms between free Banach lattices.
\begin{thm}\label{t:latticeiso}
Let $T\colon\fbl^p[E] \rightarrow \fbl^p[F]$ be a lattice homomorphism for some  $1\le p<\infty$ and some  Banach spaces~$E$ and~$F$. Then there is a unique map $\Phi_T\colon F^*\rightarrow E^*$ such that 
\begin{equation}\label{t:latticeiso:eq1}
    Tf=f\circ \Phi_T\qquad (f\in \fbl^p[E]).
\end{equation} 
This map $\Phi_T$ is positively homogeneous and satisfies
\begin{enumerate}[label={\normalfont{(\roman*)}}]
\item\label{t:latticeiso3} $\lVert \Phi_T\rVert_p= \lVert T\rVert;$
\item\label{t:latticeiso2} 
$\Phi_T{\upharpoonright_{B_{F^*}}}$ is weak*-to-weak* continuous.
\end{enumerate}
\end{thm}

\begin{proof} We begin with the uniqueness. Suppose that $\Phi_T\colon F^*\rightarrow E^*$ is a map which satisfies~\eqref{t:latticeiso:eq1}. Substituting $f=\delta_x$ for some $x\in E$ into this identity and evaluating it at some $y^*\in F^*$, we obtain
\begin{equation}\label{t:latticeiso:eq2} (T\delta_x)(y^*) = (\delta_x\circ\Phi_T)(y^*) = \Phi_T (y^*)(x). \end{equation}
Reading this equation from the right to the left, we conclude that there is only one possible way to define~$\Phi_T$. Taking this as our definition, we must check that $\Phi_T(y^*)\in E^*$, which is easy: Linearity follows from the fact that $\delta_{\lambda w+x} = \lambda\delta_w+\delta_x$ for $\lambda\in\R$ and $w,x\in E$, and continuity from the estimate $\lvert(T\delta_x)(y^*)\rvert\le \lVert T\rVert\,\lVert x\rVert\,\lVert y^*\rVert$. This estimate also shows that~$\Phi_T$ maps~$B_{F^*}$ into~$\lVert T\rVert B_{E^*}$.

 It is equally easy to see that $\Phi_T$ is positively homogeneous because 
\[ 
\Phi_T (\lambda y^*) (x)=(T\delta_{x})(\lambda y^*)=\lambda (T\delta_{x})(y^*)=\lambda \Phi_T(y^*)(x)\quad (\lambda\geq0,\,y^*\in F^*,\,x\in E). \]

Next, we shall prove~\ref{t:latticeiso3}--\ref{t:latticeiso2}, beginning with the latter. Suppose that $(y^*_\alpha)$ is a net in~$B_{F^*}$ which weak*-converges to some~$y^*\in B_{F^*}$, and take $x\in E$. Then  Lemma~\ref{L5.1}\ref{L5.1iii} implies that the net $((T\delta_{x})(y^*_\alpha))$ converges to $(T\delta_{x})(y^*)$, which in view of the definition~\eqref{t:latticeiso:eq2} means that the net
$(\Phi_T (y^*_\alpha) (x))$ converges to~$\Phi_T (y^*)(x)$.

To verify~\ref{t:latticeiso3}, we observe that the uniqueness part of the universal property of~$\fbl^p[E]$ implies that~$T$ is the lattice homomorphism induced by the bounded linear operator \mbox{$T\circ\delta^E\colon E\to\fbl^p[F]$}; that is, $T=\widehat{T\circ\delta^E}$, and therefore 
\begin{align*}
\lVert T\rVert &= \lVert T\circ\delta^E\rVert = \sup\bigl\{\lVert T\delta_x\rVert_{\fbl^p[F]} : x\in B_E\bigr\}
=\sup\biggl\{\Bigl(\sum_{j=1}^{m}\lvert (T\delta_x)(y^*_j)\rvert^p\Bigr)^{\frac1p} :\biggr.\\ &\mbox{}\hspace{3.3cm}\biggl. x\in B_E,\, m\in\mathbb N,\, (y_j^*)_{j=1}^m\subset F^*,\,\lVert (y^*_j)_{j=1}^{m}\rVert_{p,{\normalfont{\text{weak}}}}\leq 1\biggr\}\\
&=\sup\left\{\bigl\lVert\bigl(\Phi_T (y^*_j)\bigr)_{j=1}^m\bigr\rVert_{p,{\normalfont{\text{weak}}}} :  m\in\mathbb N,\, (y_j^*)_{j=1}^m\subset F^*,\, \lVert (y^*_j)_{j=1}^{m}\rVert_{p,{\normalfont{\text{weak}}}}\leq 1\right\}\\
&=\lVert \Phi_T\rVert_p.
\end{align*}

It remains to verify that the map~$\Phi_T$ defined by~\eqref{t:latticeiso:eq2} 
satisfies~\eqref{t:latticeiso:eq1}. We have already shown that~$\Phi_T$ is positively homogeneous and satisfies $\lVert\Phi\rVert_p<\infty$, so  Lemma~\ref{L:compositionOps} implies that it induces a lattice homomorphism $C_{\Phi_T}\colon H^p[E]\to H^p[F]$ by com\-po\-si\-tion; that is, $C_{\Phi_T}(f) = f\circ\Phi_T$ for $f\in H^p[E]$. The fact that 
\[ C_{\Phi_T}(\delta_x)(y^*)=\delta_x(\Phi_T (y^*)) = \Phi_T(y^*)(x) = (T\delta_{x})( y^*)\qquad (x\in E,\, y^*\in F^*) \]
shows that the restriction of~$C_{\Phi_T}$ to~$\fbl^p[E]$ satisfies 
$C_{\Phi_T}{\!\upharpoonright_{\fbl^p[E]}}\circ\delta^E=\iota\circ T\circ\delta^E$, 
where $\iota\colon\fbl^p[F]\to H^p[F]$ denotes the inclusion map.
This in turn implies that $C_{\Phi_T}{\!\upharpoonright_{\fbl^p[E]}}=\iota\circ T$ by \cite[Corollary~3.5(ii)]{JLTTT}, bearing in mind that the map~$\delta^E$ is denoted~$\phi_E^{\mathcal{D}}$ in \cite{JLTTT}. Now the conclusion follows because
\[  Tf = C_{\Phi_T}(f) = f\circ\Phi_T\qquad (f\in\fbl^p[E]). \qedhere  \]
\end{proof}

\begin{defn}
    Given a lattice homomorphism $T\colon\fbl^p[E] \rightarrow \fbl^p[F]$, we call the unique map $\Phi_T\colon F^*\to E^*$ satisfying~\eqref{t:latticeiso:eq1} the \emph{map induced by~$T$.}
\end{defn}

\begin{rem}\label{R:inducedHom} Suppose that $T\colon\fbl^p[E] \rightarrow \fbl^p[F]$ is a lattice homomorphism for some  $1\le p<\infty$ and some  Banach spaces~$E$ and~$F$, and let \mbox{$\Phi_T\colon F^*\rightarrow E^*$} be the induced map. Theorem~\ref{t:latticeiso} shows that~$\Phi_T$ is positively homogeneous, \mbox{$\lVert\Phi_T\rVert_p<\infty$} and $\Phi_T{\upharpoonright_{B_{F^*}}}$ is weak*-to-weak* continuous, so Lemma~\ref{L:compositionOps} 
implies that the com\-po\-si\-tion operator~$C_{\Phi_T}\colon H[E]\to H[F]$ induces a lattice homomorphism between each of the following four pairs of Banach lattices, for every $q\in[p,\infty)$: 
\begin{equation}\label{R:inducedHom:eq3}
\begin{alignedat}{2} H^q[E]&\rightarrow H^q[F],\qquad & H^q_{w^*}[E]&\rightarrow H^q_{w^*}[F],\\ \overline{J^q_{w^*}[E]}&\rightarrow \overline{J^q_{w^*}[F]},\qquad & M^q_{w^*,0}[E]&\rightarrow M^q_{w^*,0}[F]. \end{alignedat}
\end{equation}
In fact, we already used the first of these identities for $q=p$ in the proof of Theorem~\ref{t:latticeiso}. As we saw towards the end of that proof, we can interpret~\eqref{t:latticeiso:eq1} as the statement that the lattice homomorphism~$C_{\Phi_T}$ is an extension of~$T$ because
\begin{equation}\label{R:inducedHom:eq2}
    C_{\Phi_T}(f) = f\circ\Phi_T = Tf\qquad (f\in\fbl^p[E]). 
\end{equation}
\end{rem}

We shall now complement~\eqref{R:inducedHom:eq3} by showing that the answer to Question~\ref{Q:Iw*Invariant} is positive when $\Phi=\Phi_T$ is induced by a lattice homomorphism  $T\colon\fbl^p[E] \rightarrow \fbl^p[F]$. Furthermore, using Lemma~\ref{L:compositionOps}\ref{L:compositionOps:iii}, we can show that~$\fbl^q$ is $C_{\Phi_T}$\nobreakdash-in\-variant for every $q>p$.

\begin{lem}\label{L:Iw*invariant}
Let $T\colon\fbl^p[E] \rightarrow \fbl^p[F]$ be a lattice homomorphism for some  $1\le p<\infty$ and some  Banach spaces~$E$ and~$F$. Then:
\begin{enumerate}[label={\normalfont{(\roman*)}}]
\item\label{L:Iw*invariant:i} $C_{\Phi_T}(\overline{I_{w^*}[E]})\subseteq\overline{I_{w^*}[F]};$
\item\label{L:Iw*invariant:ii} $C_{\Phi_T}(\fbl^q[E])\subseteq\fbl^q[F]$ for every
 $q\in(p,\infty)$.
\end{enumerate}
\end{lem}

\begin{proof}
\ref{L:Iw*invariant:i}. It suffices to show that $C_{\Phi_T}(f)\in\overline{I_{w^*}[F]}$ for every $f\in I_{w^*}[E]_+$ because~$C_{\Phi_T}$ is a lattice homomorphism (in particular continuous). Take $m\in\N$ and $x_1,\ldots,x_m\in E$ such that $f\le\sum_{j=1}^m\lvert\delta_{x_j}\rvert$. Using that~$C_{\Phi_T}$ is a lattice homomorphism together with~\eqref{R:inducedHom:eq2}, we obtain
\[ 0\le C_{\Phi_T}(f)\le C_{\Phi_T}\Bigl(\sum_{j=1}^m\lvert\delta_{x_j}\rvert\Bigr) = \sum_{j=1}^m\lvert C_{\Phi_T}(\delta_{x_j})\rvert = \sum_{j=1}^m\lvert T\delta_{x_j}\rvert\in\fbl^p[F]\subseteq\overline{I_{w^*}[F]}. \]
Since $\overline{I_{w^*}[F]}$ is an ideal of~$H_{w^*}^p[F]$ and $C_{\Phi_T}(f)\in H_{w^*}^p[F]$ by~\eqref{R:inducedHom:eq3}, we conclude that  $C_{\Phi_T}(f)\in \overline{I_{w^*}[F]}$, as desired. 

\ref{L:Iw*invariant:ii}. Suppose that $1\le p < q<\infty$. Then $\lVert f\rVert_{\fbl^q[F]}\le \lVert f\rVert_{\fbl^p[F]}$ for every $f\in H[F]$ by \cite[Proposition~7.3]{JLTTT} (or Lemma~\ref{L:compositionOps}\ref{L:compositionOps:iii} in the special case $\Phi=f\colon F^*\to\R$).  This implies that $\fbl^p[F]\subseteq\fbl^q[F]$ and therefore, using~\eqref{R:inducedHom:eq2}, we obtain
\[ C_{\Phi_T}(\delta_x) = T\delta_x\in \fbl^p[F]\subseteq\fbl^q[F]\qquad (x\in E). \]
Now the conclusion follows by viewing~$C_{\Phi_T}$ as a lattice homomorphism from~$H^q[E]$ to~$H^q[F]$, as shown in~\eqref{R:inducedHom:eq3}. 
\end{proof}

\begin{cor}\label{c:extension to H0}
Let~$E$ and~$F$ be Banach spaces for which $\fbl^p[E]$ and $\fbl^p[F]$ are lattice isomorphic for some $1\le p<\infty$. Then so are each of the following six pairs for every $q\in [p,\infty)$ (where the closures are taken with respect to the $\fbl^q$-norm): 
\begin{alignat*}{3}
\fbl^q[E]&\cong\fbl^q[F];\qquad &
\overline{I_{w^*}[E]}&\cong \overline{I_{w^*}[F]};\qquad & H^q_{w^*}[E]&\cong H^q_{w^*}[F]\\ 
\overline{J^q_{w^*}[E]}&\cong\overline{J^q_{w^*}[F]};\qquad & M^q_{w^*,0}[E]&\cong M^q_{w^*,0}[F];\qquad & H^q[E]&\cong H^q[F].
\end{alignat*}
\end{cor}

\begin{proof}
Suppose that $T\colon \fbl^p[E]\rightarrow \fbl^p[F]$  is a lattice isomorphism  with inverse  $T^{-1}\colon\fbl^p[F]\rightarrow \fbl^p[E]$. 
Then, for $x\in E$ and $x^*\in E^*$, we have 
\[ x^*(x) = (T^{-1}(T\delta_x))(x^*) = (\delta_x\circ\Phi_T\circ\Phi_{T^{-1}})(x^*) = (\Phi_T\circ\Phi_{T^{-1}})(x^*)(x), \]
with a similar calculation showing that $\Phi_{T^{-1}}\circ\Phi_T = I_{F^*}$; that is, 
the induced maps $\Phi_T\colon F^*\to E^*$ and $\Phi_{T^{-1}}\colon E^*\to F^*$ are inverses of each other, 
and therefore the composition operators~$C_{\Phi_T}$ and~$C_{\Phi_{T^{-1}}}$ are also inverses of each other. As we saw in~\eqref{R:inducedHom:eq3} and Lemma~\ref{L:Iw*invariant}, they restrict to lattice homomorphisms between each of specified pairs, from which the conclusions follow. (In the case of the $\fbl^q$\nobreakdash-closure of the ideal~$I_{w^*}$ for $q>p$, a little extra care is required; before invoking Lemma~\ref{L:Iw*invariant}\ref{L:Iw*invariant:i}, we use the second part of the same result to deduce that~$C_{\Phi_T}$ is a lattice isomorphism between~$\fbl^q[E]$ and~$\fbl^q[F]$ with inverse~$C_{\Phi_{T^{-1}}}$.) 
\end{proof}

\begin{example}\label{Ex:fbl_iso}
  One cannot in general reverse the implication 
  \begin{equation*} 
  \fbl^p[E]\cong\fbl^p[F]\ \implies\ H^p[E]\cong H^p[F] 
  \end{equation*}
  proved in Corollary~\ref{c:extension to H0}. To see this, take Banach spaces~$E$ and~$F$ whose dual spaces~$E^*$ and~$F^*$ are linearly isomorphic. Then it is clear that~$H^p[E]$ and~$H^p[F]$ are lattice isomorphic. In particular,  $H^1[\ell_1]$ is lattice isomorphic to $H^1[L_1[0,1]]$, but $\fbl^1[\ell_1]$ and $\fbl^1[L_1[0,1]]$ are not lattice isomorphic by \cite[Theorems~4.11 and~4.13]{ART}.

Proposition~\ref{prop:findimcase} provides another example, valid for any $1\le p<\infty$; it shows that~$H^p[E]$ and~$H^p[F]$ are lattice iso\-mor\-phic for every pair~$E$ and~$F$ of finite-dimensional Banach spaces of di\-men\-sion at least~$2$, but~$\fbl^p[E]$ and~$\fbl^p[F]$ are only lattice isomorphic if $\dim E=\dim F$. 
\end{example}

Our next result uses Lemma~\ref{L:compositionOps} to generalize the above observation that~$H^p[E]$ and~$H^p[F]$ are lattice iso\-mor\-phic when~$E$ and~$F$ have finite dimension at least~$2$
to the infinite-dimensional case.

\begin{prop}\label{p:r2r3} Let~$F$ be a closed subspace of finite codimension in a Banach space~$E$, and suppose that~$F$ has dimension at least~$2$. Then the Banach lattices~$H^p[E]$ and~$H^p[F]$ are lattice isomorphic for every $1\le p<\infty$.
\end{prop}

\begin{proof} Since $H^p[G_1]$ and $H^p[G_2]$ are lattice isomorphic whenever~$G_1$ and~$G_2$ are linearly isomorphic Banach spaces, it suffices to show that~$H^p[E]$ and~$H^p[F]$ are lattice isomorphic for $E= \ell_1^m\oplus G$ and $F= \ell_1^n\oplus G$, where~$G$ is a Banach space and  $m,n\in\N\setminus\{1\}$. 
Take a bijection $\theta\colon S_{\ell_\infty^n}\rightarrow S_{\ell_\infty^m}$, where $S_{\ell_\infty^k}$ denotes the unit sphere of \mbox{$\ell_\infty^k = (\ell_1^k)^*$}, as usual. (Such a bijection exists because both sets have cardinality~$\mathfrak{c}$.) 
Let $\Theta\colon \ell_\infty^n\rightarrow \ell_\infty^m$ be the positively homogeneous extension of~$\theta$, that is, $\Theta(0)=0$ and 
\begin{equation*}
\Theta(y^*)=\lVert y^*\rVert_\infty\,\theta\Bigl(\frac{y^*}{\lVert y^*\rVert_\infty}\Bigr) = \bigvee_{i=1}^n\lvert\delta_{e_i}(y^*)\rvert\,\theta\Bigl(\frac{y^*}{\lVert y^*\rVert_\infty}\Bigr)\qquad (y^*\in\ell^n_\infty\setminus\{0\}), \end{equation*}
where $(e_i)_{i=1}^n$ denotes the unit vector basis for~$\ell_1^n$. 

To verify that $\lVert\Theta\rVert_p<\infty$, take $(y_j^*)_{j=1}^k\subset\ell_\infty^n$ with $\bigl\lVert (y_j^*)_{j=1}^k\bigr\rVert_{p,\text{weak}}\le 1$ for some $k\in\N$. We may suppose that $y_j^*\ne 0$ for each $j\in\{1,\ldots,k\}$, and find
\begin{align*}
\bigl\lVert \bigl(\Theta(y_j^*)\bigr)_{j=1}^k\bigr\rVert_{p,\text{weak}} &= \sup_{x\in B_{\ell_1^m}}\biggl(\sum_{j=1}^k\Bigl(\bigvee_{i=1}^n\lvert\delta_{e_i}(y^*_j)\rvert\Bigr)^p\,\Bigl|\Bigl\langle \theta\Bigl(\frac{y^*_j}{\lVert y^*_j\rVert_\infty}\Bigr),x\Bigr\rangle\Bigr|^p\biggr)^{\frac1p}\\
&\le \biggl(\sum_{j=1}^k\Bigl(\bigvee_{i=1}^n\lvert\delta_{e_i}(y^*_j)\rvert\Bigr)^p\biggr)^{\frac1p} \le \Bigl\lVert \bigvee_{i=1}^n\lvert\delta_{e_i}\rvert\Bigr\rVert_{\fbl^p[\ell_1^n]}
\end{align*}
where the first inequality follows from the fact that $\bigl\lVert\theta(y^*_j/\lVert y^*_j\rVert_\infty)\bigr\rVert_\infty = 1$.
Hence \mbox{$\lVert\Theta\rVert_p\le \bigl\lVert \bigvee_{i=1}^n\lvert\delta_{e_i}\rvert\bigr\rVert_{\fbl^p[\ell_1^n]}<\infty$}. 

We can obviously extend~$\Theta$ to a positively homogeneous map $\Phi\colon  F^*\to E^*$ by defining $\Phi(y^*,z^*) = (\Theta(y^*),z^*)$ for $y^*\in\ell_\infty^n$ and $z^*\in G^*$.
This extension satisfies $\lVert \Phi\rVert_p\le \lVert \Theta\rVert_p +1<\infty$ because 
\begin{align*}
  \bigl\lVert\bigl(\Phi(y_j^*,z^*_j)\bigr)_{j=1}^k\bigr\rVert_{p,\text{weak}} &= \bigl\lVert\bigl(\Theta(y_j^*),z^*_j\bigr)_{j=1}^k\bigr\rVert_{p,\text{weak}}\\
  &\le \bigl\lVert \bigl(\Theta(y_j^*)\bigr)_{j=1}^k\bigr\rVert_{p,\text{weak}} + \lVert (z^*_j)_{j=1}^k\rVert_{p,\text{weak}}\le \lVert \Theta\rVert_p +1
\end{align*}
for every $k\in\N$ and $(y_j^*,z^*_j)_{j=1}^k\subset F^*$ with $\lVert (y_j^*,z^*_j)_{j=1}^k\rVert_{p,\text{weak}}\le 1$. 

On the other hand, by repeating the above arguments for the inverse map $\theta^{-1}\colon S_{\ell_\infty^m}\rightarrow S_{\ell_\infty^n}$ instead of~$\theta$, we obtain a positively homogeneous map \mbox{$\Phi^{-1}\colon E^*\to F^*$} which is an inverse of~$\Phi$ (as our choice of symbol for it indicates) and satisfies  
\[ \lVert \Phi^{-1}\rVert_p\le \biggl\lVert \bigvee_{i=1}^m\lvert\delta_{e_i}\rvert\biggr\rVert_{\fbl^p[\ell_1^m]} + 1 <\infty. \] 
Therefore, by~\eqref{L:compositionOps:eq1}, we may regard the restrictions of~$C_\Phi$ and $C_{\Phi^{-1}}$ as lattice homo\-mor\-phisms $C_\Phi\colon H^p[E]\to H^p[F]$ and $C_{\Phi^{-1}}\colon H^p[F]\to H^p[E]$, respectively. They are clearly inverses of each other, from which the conclusion follows. 
\end{proof}

\begin{example}\label{Ex:GM} 
Gowers~\cite{gowers} and Gowers and Maurey~\cite{gm1} have shown that there are infinite-dimensional Banach spaces~$E$ which are not isomorphic to their hyper\-planes. However,  Proposition~\ref{p:r2r3} shows that~$H^p[E]$ and~$H^p[F]$ are lattice isomorphic whenever~$F$ is a closed subspace of finite codimension in~$E$ and $1\le p<\infty$. 

More generally, for every $k\in\{0,2,3,\ldots\}$, Gowers  and Maurey~\cite{gm2} have constructed a Banach space~$E_k$ such that~$E_k$ is isomorphic to a closed  subspace~$F\subset E_k$ if and only if~$F$ has finite codimension in~$E_k$, and this codimension is a multiple of~$k$. As above, $H^p[E_k]$ and~$H^p[F]$ are lattice isomorphic for every $1\le p<\infty$ and every closed sub\-space~$F$ of finite codimension in~$E_k$, irrespective of the value of this co\-di\-men\-sion. 
\end{example}

\begin{rem}
Examples~\ref{Ex:fbl_iso} and~\ref{Ex:GM} provide examples of non-isomorphic Banach spaces~$E$ and~$F$ for which $H^p[E]$ and~$H^p[F]$ are lattice isomorphic for every \mbox{$1\le p<\infty$}. As the proof of Proposition~\ref{p:r2r3} shows, the lattice isomorphisms in Example~\ref{Ex:GM} are induced by composition with a positively homogeneous bijection $\Phi\colon F^*\to E^*$,  which satisfies $\lVert\Phi\rVert_p<\infty$, but lacks any weak* continuity properties. 

In general, exhibiting weak* homeomorphisms between the unit balls of dual Banach spaces is not difficult. For instance, Keller's Theorem (see, \emph{e.g.,} \cite[Section~12.3]{fab-ultimo}) implies that the dual unit ball $B_{E^*}$ is weak* homeomorphic to the Hilbert cube $[-1, 1]^{\N}$ for every separable, infinite-dimensional Banach space~$E$. Under some technical assumptions, this weak* homeomorphism can be chosen to be positively homo\-geneous, as shown in \cite[Theorem 10.24]{OTTT}. 

Unfortunately, none of the above arguments provide any information about the corresponding free Banach lattices. It remains a major open problem whether there exist non-isomorphic Banach spaces~$E$ and~$F$ such that $\fbl^p[E]$ is lattice isomorphic to $\fbl^p[F]$ for some $1\le p<\infty$.
\end{rem}

The next lemma 
relates the properties of a lattice homomorphism between free Banach lattices to the map it induces.

\begin{lem}\label{l:phi injective}
Let $T\colon\fbl^p[E] \rightarrow \fbl^p[F]$ be a lattice homomorphism for some  $1\le p<\infty$ and some  Banach spaces~$E$ and~$F$. Then:
\begin{enumerate}[label={\normalfont{(\roman*)}}]
\item\label{l:phi injective:i} $T$ has dense range if and only if the sublattice generated by the range of the operator~$T\circ\delta^E$ is dense in~$\fbl^p[F]$. 
\item\label{l:phi injective:ii} Suppose that one and hence both of the conditions in~\ref{l:phi injective:i} are satisfied. Then the induced map $\Phi_T\colon F^*\to E^*$ is injective.
\end{enumerate}
\end{lem}

\begin{proof}
\ref{l:phi injective:i} follows from the fact that~$\fbl^p[E]$ is the closed sublattice of~$H^p[E]$ generated by the image of~$\delta^E$.

\ref{l:phi injective:ii}. Suppose that the sublattice generated by $(T\circ\delta^E)(E)$ is dense in~$\fbl^p[F]$. Since~$\Phi_T$ is positively homogeneous, it suffices to verify that its restriction to~$B_{F^*}$ is injective. Take $y^*,z^*\in B_{F^*}$ such that $\Phi_T(y^*)=\Phi_T(z^*)$.  By~\eqref{t:latticeiso:eq1}, we have 
\[ (T\delta_x)(y^*) = \delta_x(\Phi_T(y^*)) = \delta_x(\Phi_T(z^*)) = (T\delta_x)(z^*)\qquad (x\in E). \] 
This implies that $f(y^*)=f(z^*)$ for every~$f$ belonging to the sublattice generated by \mbox{$(T\circ\delta^E)(E)$}. By hypothesis, this sub\-lattice is dense in~$\fbl^p[F]$ with respect to the~$\fbl^p$-norm, which according to Lemma~\ref{L5.1}\ref{L5.1vi'} dominates the uniform norm on~$B_{F^*}$, so we conclude that $f(y^*)=f(z^*)$ for every $f\in \fbl^p[F]$. In particular, taking $f=\delta_y$ for $y\in F$, we see that $y^*(y) = z^*(y)$, which proves that $y^*= z^*$. 
\end{proof}

\begin{example}\label{OpInducedLatHom} Given a bounded linear operator $S\colon E\to F$ between Banach spaces~$E$ and~$F$, there is a standard way to associate a lattice homomorphism between the corresponding free Banach lattices with it.  Indeed, for \mbox{$1\le p<\infty$}, $\delta^F\circ S\colon E\to\fbl^p[F]$ is a bounded linear operator of norm~$\lVert S\rVert$ into a $p$-convex Banach lattice, so the universal property of~$\fbl^p[E]$ implies that there is a unique lattice homomorphism $\overline{S} = \widehat{\delta^F\circ S}\colon \fbl^p[E]\to\fbl^p[F]$ such that 
\begin{equation}\label{OpInducedLatHom:eq3} 
\overline{S}\circ\delta^E = \delta^F\circ S, 
\end{equation}
and $\lVert \overline{S}\rVert = \lVert S\rVert$. The map $\Phi_{\overline{S}}\colon F^*\to E^*$ induced by~$\overline{S}$ is simply~$S^*$, as the following calculation, valid for every $y^*\in F^*$ and $x\in E$, shows:
\begin{align*}
    \Phi_{\overline{S}}(y^*)(x) &= (\delta_x\circ\Phi_{\overline{S}})(y^*) = (\overline{S}\delta_x)(y^*) = \delta_{Sx}(y^*) = y^*(Sx) = S^*(y^*)(x).
\end{align*}
Consequently, we have 
\begin{equation}\label{OpInducedLatHom:eq2}
\overline{S}f = f\circ S^* = C_{S^*}(f)\qquad (f\in\fbl^p[E]), 
\end{equation}
and  Lemma~\ref{L:Iw*invariant}\ref{L:Iw*invariant:i} implies that
    \begin{equation}\label{OpInducedLatHom:eq1}
        C_{S^*}(\overline{I_{w^*}[E]})\subseteq \overline{I_{w^*}[F]}.
    \end{equation} 
\end{example}

We can use this inclusion to obtain stability results for the properties that  Questions~\ref{Q5.7} and~\ref{Q5.8} are concerned with. 

\begin{prop}\label{lem:fbl_neq_clIw2} Let $1\le p <\infty$, and suppose that~$F$ is a complemented subspace of a Banach space~$E$ for which $\fbl^p[E]=\overline{I_{w^*}[E]}$. 
Then $\fbl^p[F]=\overline{I_{w^*}[F]}$. 
\end{prop}

\begin{proof} The hypothesis that~$F$ is complemented in~$E$ means that the inclusion map $J\colon F\hookrightarrow E$ has a bounded linear left inverse  $L\colon E\to F$; that is, $LJ = I_F$. This implies that 
$(C_{L^*}\circ C_{J^*})f = f\circ J^*\circ L^* = f$ for every $f\in H[F]$, and consequently
\begin{align*} \overline{I_{w^*}[F]} = C_{L^*}(C_{J^*}(\overline{I_{w^*}[F]})) \subseteq C_{L^*}(\overline{I_{w^*}[E]}) &= C_{L^*}(\fbl^p[E])\\ &= \overline{L}(\fbl^p[E])\subseteq \fbl^p[F],
\end{align*}
where we have used~\eqref{OpInducedLatHom:eq1}, the hypothesis and~\eqref{OpInducedLatHom:eq2}. 
This completes the proof because the opposite inclusion $\fbl^p[F]\subseteq\overline{I_{w^*}[F]}$ is always true. 
\end{proof}

\begin{rem}
    In view of Theorem~\ref{thm:fbl_neq_clIw_p},  it is possible that Proposition~\ref{lem:fbl_neq_clIw2} has no genuine content in the sense that the hypothesis that $\fbl^p[E]=\overline{I_{w^*}[E]}$ for some $1\le p <\infty$ may only be satisfied when~$E$ is  fi\-nite-di\-men\-sional.
\end{rem}

\begin{prop}\label{P:Q4.2forQuotients}
Let~$E$ be a reflexive Banach space for which $\overline{I_{w^*}[E]} = H_{w^*}^p[E]$ for some $1\le p<\infty$. Then $\overline{I_{w^*}[F]} = H_{w^*}^p[F]$ for every quotient space~$F$ of~$E$. 
\end{prop}

\begin{proof}
Let $Q\colon E\rightarrow F$ denote the quotient map, and write $q\colon B_E\to B_F$ for its restriction to the unit balls. Then~$q$ is continuous with respect to the relative weak topologies and surjective (because~$E$ and~$F$ are reflexive), so we can define an isometric lattice homomorphism $C_q\colon C(B_F)\to C(B_E)$ by composition: $C_qf = f\circ q$. Hence, by \cite[Theorem~1.4.19]{MN}, its adjoint $C_q^*\colon C(B_E)^*\to C(B_F)^*$ is a surjective, interval-preserving operator.

For every $f\in H_{w^*}^p[F]$, \cite[Proposition~7.7]{JLTTT} implies that
\begin{equation}\label{P:Q4.2forQuotients:eq1}
  \lvert f(x^*)\rvert\leq \lVert f\rVert_{\fbl^p[F]}\,f_\mu^p(x^*)\qquad (x^*\in F^*)
\end{equation} 
for some probability measure $\mu\in C(B_{F})^*_+$,
where~$f_\mu^p\in H^p[F]$ is the function defined by~\eqref{eq:defnfmu}. 
In view of the previous paragraph, we can take $\nu\in C(B_E)_+^*$ such that $C_q^*(\nu)=\mu$. Since~$E$ is reflexive,  $f_\nu^p\in H_{w^*}^p[E] = \overline{I_{w^*}[E]}$ by Theorem~\ref{t:ref} and the hypothesis..

For $y^*\in B_{F^*}$, the function $\psi_p(y^*)\in C(B_F)$ defined in Lemma~\ref{lem:reflexive} satisfies
\[ C_q(\psi_p(y^*))(x) = \psi_p(y^*)(q(x)) = \lvert y^*(q(x))\rvert^p = \lvert (Q^*y^*)(x)\rvert^p = \psi_p(Q^*y^*)(x) \]
for each $x\in B_E$, so 
\begin{align*}
    \bigl(f_\mu^p(y^*)\bigr)^p &= \langle  \psi_p(y^*),\mu\rangle = \langle C_q(\psi_p(y^*)),\nu\rangle\\ &= \langle\psi_p(Q^*y^*),\nu\rangle = \bigl(f_\nu^p(Q^*y^*)\bigr)^p = \bigl(C_{Q^*}(f_\nu^p)(y^*)\bigr)^p. 
\end{align*}
This proves that $f_\mu^p = C_{Q^*}(f_\nu^p)$, which belongs to~$\overline{I_{w^*}[F]}$ by~\eqref{OpInducedLatHom:eq1}, and therefore  $f\in\overline{I_{w^*}[F]}$ by~\eqref{P:Q4.2forQuotients:eq1} because~$\overline{I_{w^*}[F]}$ is an ideal of~$H_{w^*}^p[F]$.
\end{proof}

We shall next address two very natural questions concerning a lattice homomorphism~$T$ between free Banach lattices: when is the induced map~$\Phi_T$ linear? And how can we tell whether~$T$ arises from a bounded linear operator between the underlying Banach spaces via the construction described in Example~\ref{OpInducedLatHom}?

\begin{prop}\label{p:philinear}
Let $T\colon\fbl^p[E] \rightarrow \fbl^p[F]$ be a lattice homomorphism for some  $1\le p<\infty$ and some  Banach spaces~$E$ and~$F$. The following conditions are equivalent: 
\begin{enumerate}[label={\normalfont{(\alph*)}}]
\item\label{p:philinear:a} The induced map $\Phi_T\colon F^*\rightarrow E^*$ is linear.
\item\label{p:philinear:c} The induced map $\Phi_T\colon F^*\rightarrow E^*$ is an adjoint operator; that is, $\Phi_T=S^*$ for some bounded linear operator $S\colon E\rightarrow F$.
\item\label{p:philinear:b} There is a bounded linear operator $S\colon E\rightarrow F$ for which  $T\circ\delta^E=\delta^F\circ S$.
\item\label{p:philinear:d} $T=\overline{S}$ for some bounded linear operator $S\colon E\rightarrow F$.
\end{enumerate}
\end{prop}

\begin{proof}  We begin by showing that conditions~\ref{p:philinear:a} and~\ref{p:philinear:c} are equivalent. It is clear that~\ref{p:philinear:c} implies~\ref{p:philinear:a}. Conversely, suppose that~$\Phi_T$ is linear. By Theorem~\ref{t:latticeiso}, its restriction~$\Phi_T{\upharpoonright_{B_{F^*}}}$ is weak*-to-weak* continuous. This implies that $\Phi_T=S^*$ for some bounded linear operator $S\colon E\rightarrow F$, as one can prove by adding an application of the Krein--Smu\-lian Theorem to the standard proof that a weak*-to-weak* continuous linear operator between dual Banach spaces is the adjoint of a bounded linear operator. 

To see that conditions~\ref{p:philinear:c} and~\ref{p:philinear:b} are also equivalent, we observe that for $x\in E$, $y^*\in F^*$ and a bounded linear operator $S\colon E\rightarrow F$, we have 
\[ (\delta^F\circ S)(x)(y^*) = \delta_{Sx}(y^*) = y^*(Sx) = (S^*y^*)(x), \] 
while~\eqref{t:latticeiso:eq1} implies that
\[ (T\circ\delta^E)(x)(y^*) = 
(T\delta_x)(y^*) = (\delta_x\circ\Phi_T)(y^*) = \Phi_T(y^*)(x). \]
Comparing these two equations, we conclude that $T\circ\delta^E=\delta^F\circ S$ if and only if $\Phi_T=S^*$.

Finally, we verify that~\ref{p:philinear:b} and~\ref{p:philinear:d}  are equivalent. 
Equation~\eqref{OpInducedLatHom:eq3} shows that~\ref{p:philinear:d} implies~\ref{p:philinear:b}, while the converse implication follows from the fact that~$\overline{S}$ is the unique lattice homomorphism satisfying~\eqref{OpInducedLatHom:eq3}.
\end{proof}

The following example shows that not all lattice homomorphisms between free Banach lattices arise from bounded linear operators, or equivalently that the maps they induce need not be linear. 

\begin{example}\label{Ex:modulus}
Take $1\le p<\infty$, let $(e_n)_{n\in\N}$ denote the unit vector basis for~$\ell_1$, and identify~$\ell_1^*$ with~$\ell_\infty$, as usual. Since 
\[ \bigl\lVert\lvert\delta_{e_n}\rvert\bigr\rVert_{\fbl^p[\ell_1]} = \lVert\delta_{e_n}\rVert_{\fbl^p[\ell_1]} = \lVert e_n\rVert_{\ell_1} = 1\qquad (n\in\N), \] 
we can define a bounded linear operator $R\colon\ell_1\to\fbl^p[\ell_1]$ of norm~$1$ by $Re_n = \lvert\delta_{e_n}\rvert$ for every \mbox{$n\in\N$}. The universal property of~$\fbl^p[\ell_1]$ implies that~$R$ induces a lattice homomorphism $\widehat{R}\colon\fbl^p[\ell_1]\to\fbl^p[\ell_1]$ such that $\widehat{R}\circ\delta^{\ell_1} = R$, which in turn induces a map $\Phi_{\widehat{R}}\colon\ell_\infty\to\ell_\infty$ such that $\widehat{R}f = f\circ\Phi_{\widehat{R}}$ for every  $f\in\fbl^p[\ell_1]$ by Theorem~\ref{t:latticeiso}. Combining these identities with the definition of~$R$, we obtain
\begin{align*}  \Phi_{\widehat{R}}(x^*)(e_n) &= (\delta_{e_n}\circ\Phi_{\widehat{R}})(x^*) = (\widehat{R}\delta_{e_n})(x^*) = (Re_n)(x^*)\\ &= \lvert\delta_{e_n}\rvert(x^*) = \lvert x^*(e_n)\rvert = \lvert x^*\rvert(e_n)\qquad (n\in\N,\, x^*\in\ell_\infty).  \end{align*}
This shows that $\Phi_{\widehat{R}}(x^*)=\lvert x^*\rvert$ for every $x^*\in\ell_\infty$, so $\Phi_{\widehat{R}}$ is non-linear. 
\end{example}

Although the map induced by a lattice homomorphism between free Banach lattices need not be linear,  it is sometimes possible to find a bounded linear operator which agrees with it on certain vectors.

\begin{prop}
Let $\Phi\colon F^*\rightarrow E^*$ be a positively homogeneous map for some Banach spaces~$E$ and~$F$, and suppose that $\lVert\Phi\rVert_1<\infty$ and~$F^*$ has an unconditional basis $(f_n^*)_{n\in \mathbb N}$. Then the map $R\colon f_n^*\mapsto\Phi(f_n^*)$ for $n\in\N$ extends to a bounded linear operator \mbox{$R\colon F^*\rightarrow E^*$}.
\end{prop}

\begin{proof} We extend~$R$ by linearity to the dense subspace of~$F^*$ spanned by $(f_n^*)_{n\in\N}$ and claim that this map is bounded by~$K\,\lVert\Phi\rVert_1$, where~$K$ denotes the un\-con\-di\-tional basis constant of~$(f_n^*)$.  Take $(a_j)_{j=1}^m\subset\R$ for some $m\in\N$, and choose $\sigma_j=\pm1$ such that $\sigma_ja_j\ge0$ for each $j\in\{1,\ldots,m\}$. Then~\eqref{Eq:weakpsumnorm:1} implies that
\[ \bigl\lVert(\sigma_ja_jf_j^*)_{j=1}^m\bigr\rVert_{1,\text{weak}}\le K\,\biggl\lVert\sum_{j=1}^m a_jf_j^*\biggr\rVert,  \]
so by~\eqref{defn:Phi_p_norm} and the positive homogeneity of~$\Phi$, we have 
\begin{align*}
\lVert\Phi\rVert_1K\,\biggl\lVert\sum_{j=1}^m a_jf_j^*\biggr\rVert &\ge \bigl\lVert\bigl(\Phi(\sigma_ja_jf_j^*)\bigr)_{j=1}^m\bigr\rVert_{1,\text{weak}} = \sup_{\epsilon_j=\pm1}\biggl\lVert\sum_{j=1}^m \epsilon_j\Phi(\sigma_ja_jf_j^*)\biggr\rVert\\ 
&=\! \sup_{\epsilon_j=\pm1}\biggl\lVert\sum_{j=1}^m \epsilon_j\sigma_ja_j\Phi(f_j^*)\biggr\rVert\ge 
\biggl\lVert\sum_{j=1}^m a_j\Phi(f_j^*)\biggr\rVert = \biggl\lVert R\Bigl(\sum_{j=1}^m a_jf_j^*\Bigr)\biggr\rVert. 
\end{align*}
This proves the claim, and the result follows.
\end{proof}

Perhaps a strategy to show that the Banach spaces~$E$ and~$F$ are linearly isomorphic whenever a lattice isomorphism $T\colon\fbl^p[E]\to\fbl^p[F]$ exists for some \mbox{$1\le p<\infty$}
would be to find a suitable linearization of the induced map $\Phi_T\colon F^*\rightarrow E^*$, which could then be used to construct an isomorphism between~$E$ and~$F$. 

\begin{rem}
    Suppose that $\Phi\colon F^*\rightarrow E^*$ is positively homogeneous and satisfies $\lVert\Phi\rVert_p<\infty$ for some $1\le p<\infty$. Then, for $m\in\N$ and $y_1^*,\ldots,y_m^*\in F^*$, we have
    \begin{align*}
        \biggl\lVert\Phi\Bigl(\sum_{j=1}^m y_j^*\Bigr) - \sum_{j=1}^m \Phi(y_j^*)\biggr\rVert &\le 
        \biggl\lVert\Phi\Bigl(\sum_{j=1}^m y_j^*\Bigr)\biggr\rVert + \sum_{j=1}^m \lVert\Phi(y_j^*)\rVert\\ &\le  \lVert\Phi\rVert_p  \biggl\lVert\sum_{j=1}^m y_j^*\biggr\rVert + \sum_{j=1}^m \lVert\Phi\rVert_p\lVert y_j^*\rVert\le 2\lVert\Phi\rVert_p\sum_{j=1}^m \lVert y_j^*\rVert. 
    \end{align*}
    This implies that~$\Phi$ would be quasi-linear if it was homogeneous (not just positively homogeneous). Since there is more than one definition of ``quasi-linear'' in circulation, let us state explicitly the definition we are using, following \cite[Chapter~16]{BL} and \cite[Definition~3.2.1]{C-SC}, not \cite[Definition~2.a.9]{LT2}: a map $\Phi\colon E\to F$ between Banach spaces~$E$ and~$F$ is \emph{quasi-linear} if it is homogeneous and there is a constant $K>0$ such that  
    \[ \lVert\Phi(x+y) - \Phi(x) - \Phi(y)\rVert\leq K(\lVert x\rVert + \lVert y\rVert)\qquad (x,y\in E). \]
\end{rem}

As we saw in Example~\ref{Ex:modulus}, the issue that the map~$\Phi$ may not be homogeneous is genuine, even when it is induced by a lattice homomorphism between free Banach lattices. However,  our next result shows that there are instances where homogeneity is auto\-ma\-tic, so the induced map is quasi-linear. 

\begin{lem}\label{L:1april2024}
    Let~$E$ and~$F$ be Banach spaces, and let $\Phi\colon F^*\to E^*$ be a positively homogeneous bijection for which $\lVert\Phi\rVert_1 = 1 = \lVert\Phi^{-1}\rVert_1$. Suppose that the norm on~$E^*$ is strictly convex. Then~$\Phi$ is homogeneous.
\end{lem}

\begin{proof}
Since~$\Phi$ is positively homogeneous, it suffices to show that $\Phi(y^*)=-\Phi(-y^*)$ for every unit vector $y^*\in F^*$. The triangle inequality and the fact that  $\lVert\Phi\rVert_1 = 1$ imply that  $\lVert\Phi(y^*)\pm\Phi(-y^*)\rVert\le 2$. Combining this with~\eqref{Eq:weakpsumnorm:1},     we obtain
\begin{align*} 
2&\ge\max\lVert\Phi(y^*)\pm\Phi(-y^*)\rVert = \lVert(\Phi(y^*),\Phi(-y^*))\rVert_{1,\text{weak}}\ge \lVert(y^*,-y^*)\rVert_{1,\text{weak}}
\end{align*}
because $\lVert\Phi^{-1}\rVert_1 = 1$. Another application of~\eqref{Eq:weakpsumnorm:1} shows that $\lVert(y^*,-y^*)\rVert_{1,\text{weak}} = \max\lVert y^*\pm(-y^*)\rVert = 2$,  so either $\lVert\Phi(y^*)+\Phi(-y^*)\rVert = 2$ or $\lVert\Phi(y^*)-\Phi(-y^*)\rVert = 2$. Since $\lVert\pm\Phi(\pm y^*)\rVert\le \lVert y^*\rVert = 1$, the strict convexity of the norm on~$E^*$ implies that $\Phi(y^*)=\Phi(-y^*)$ in the first case and $\Phi(y^*)=-\Phi(-y^*)$ in the second. However, $\Phi(y^*)=\Phi(-y^*)$ is impossible because~$\Phi$ is injective and $y^*\ne0$, so we must be in the second case, which gives the desired conclusion.
\end{proof}

\begin{cor}\label{C:1april2024} Let $T\colon \fbl^1[E]\rightarrow \fbl^1[F]$ be an isometric lattice isomorphism for some Banach spaces~$E$ and~$F$, and suppose that the norm on~$E^*$ is strictly convex. Then the induced map $\Phi_T\colon F^*\rightarrow E^*$ is homogeneous.
\end{cor}

\begin{proof}
This follows immediately from Lemma~\ref{L:1april2024} because the hypothesis that~$T$ is an isometric lattice isomorphism means that the induced map $\Phi_T\colon F^*\to E^*$ is a positively homogeneous bijection with inverse $\Phi_{T^{-1}}$, so in particular 
$\lVert\Phi_T\rVert_1 = \lVert T\rVert = 1$ and $\lVert\Phi_T^{-1}\rVert_1 =  \lVert T^{-1}\rVert = 1$ by Theorem~\ref{t:latticeiso}\ref{t:latticeiso3}. 
\end{proof}

We refer to \cite[Sec\-tion~10.3]{OTTT} for a detailed analysis of lattice isometries between free Banach lattices.

We conclude with an extension of another result from~\cite{OTTT}. 
By an isomorphic embedding of a Banach space~$E$ into a Banach space~$F$, we understand a bounded linear operator $J\colon E\to F$ which is injective and  has closed range, or in other words~$J$ is bounded below by some constant $\eta>0$ in the sense that $\lVert Jx\rVert\ge\eta\lVert x\rVert$ for every $x\in E$. In this case $J^*\colon F^*\to E^*$ is surjective, so~\eqref{OpInducedLatHom:eq2} implies that the associated lattice homomorphism $\overline{J}\colon\fbl^p[E]\to\fbl^p[F]$ is injective for every $1\le p<\infty$.  
The question of whether~$\overline{J}$ is an isomorphic embedding (that is, has closed range) was analyzed in depth in \cite[Section~3.3]{OTTT}. We shall next provide an extension of the main result therein. In the proof, we require the following standard identity between the operator norm of a bounded linear operator~$S$ from a Banach space~$E$ into~$\ell_p^n$ for some $1\le p<\infty$ and $n\in\N$ and the weak $p$-norm:
\begin{equation}\label{eq:opnorm_weakpnorm}
    \lVert S\rVert = \lVert (S^*e_j^*)_{j=1}^n\rVert_{p,\text{weak}},
\end{equation}
where $(e_j^*)_{j=1}^n$ denotes the unit vector basis of $(\ell_p^n)^*$. 

\begin{prop}\label{p:embed}
Let $J\colon E\hookrightarrow F$ be an isomorphic embedding of a Banach space~$E$ into a Banach space~$F$, and take $1\leq p<\infty$. The following conditions are equivalent:
\begin{enumerate}[label={\normalfont{(\alph*)}}]
\item\label{p:embed1} There is a constant $C\geq 1$ such that, for every $n\in\mathbb{N}$, every bounded linear operator $S\colon E\rightarrow\ell_p^n$ admits a bounded linear extension $\widetilde{S}\colon F\rightarrow \ell_p^n$ (in the sense that $S=\widetilde{S}J$) with $\lVert\widetilde{S}\rVert\leq C\lVert S\rVert$.
\item\label{p:embed3} The lattice homomorphism $\overline{J}$ is an isomorphic embedding of~$\fbl^p[E]$ into $\fbl^p[F]$.
\item\label{p:embed2} The lattice homomorphism $C_{J^*}$ is an isomorphic embedding of~$H^p[E]$ into $H^p[F]$.
\end{enumerate}
\end{prop}

\begin{proof}
The equivalence of conditions~\ref{p:embed1} and~\ref{p:embed3} is essentially a restatement of \cite[Theorem~3.7, (1)$\Leftrightarrow$(3)]{OTTT}. (``Essentially'' refers to the fact that   \cite[Theorem 3.7]{OTTT} is stated only for isometric embeddings. However, it is explained in the text above it how to reduce the isomorphic case to the isometric.) Furthermore, the implication \ref{p:embed2}$\Rightarrow$\ref{p:embed3} is immediate from~\eqref{OpInducedLatHom:eq2}, so it only remains to verify that~\ref{p:embed1} implies~\ref{p:embed2}. 

Suppose that~\ref{p:embed1} is satisfied for some constant~$C>0$, and take $\eta\in(0,1)$. Then, for every  $f\in H^p[E]$, we can find $n\in\N$ and $(x_j^*)_{j=1}^n\subset E^*$ with $\lVert (x^*_j)_{j=1}^n\rVert_{p,\text{weak}} =1$ such that 
\[ \eta\lVert f\rVert_{\fbl^p[E]}\leq \Bigl(\sum_{j=1}^{n} \lvert f(x_j^*)\rvert^p\Bigr)^{\frac1p}. \] 
Define an operator $S\colon E\rightarrow\ell_p^n$ by $Sx=\sum_{j=1}^{n}\langle x,x_j^*\rangle e_j$, where $(e_j)_{j=1}^{n}$ denotes the unit vector basis of~$\ell_p^n$. Clearly~$S$ is bounded and linear, and $\lVert S\rVert =1$ by~\eqref{eq:opnorm_weakpnorm} because $S^*e_j^* = x_j^*$ for each $j=1,\ldots,n$, so the hypothesis implies that~$S$ admits a bounded linear extension $\widetilde{S}\colon F\rightarrow \ell_p^n$ whose norm is at most~$C$. Another appli\-ca\-tion of~\eqref{eq:opnorm_weakpnorm} shows that $\lVert (\widetilde{S}^*e^*_j)_{j=1}^n\rVert_{p,\text{weak}}\le C$, and consequently we have
\begin{align*}
C^p\lVert C_{J^*}f\rVert_{\fbl^p[F]}^p &\geq \sum_{j=1}^n\lvert (C_{J^*}f)(\widetilde{S}^*e_j^*)\rvert^p = \sum_{j=1}^n\lvert f(J^*\widetilde{S}^*e_j^*)\rvert^p\\ &= \sum_{j=1}^n\lvert f(S^*e_j^*)\rvert^p = \sum_{j=1}^n\lvert f(x_j^*)\rvert^p\ge \eta^p\lVert f\rVert_{\fbl^p[E]}^p.
\end{align*}
This proves that~$C_{J^*}$ is bounded below by $\eta/C$, so it is an isomorphic embedding.
\end{proof}

\section*{Acknowledgements}
\noindent 
This research was initiated during a visit of the second author to Lancaster Uni\-ver\-sity and continued during visits of the first author to ICMAT, supported by the London Mathematical Society (grants number 21805 and 42109). Furthermore, Tradacete's research is partially supported by grants PID\-2020-116398GB-I00 and CEX\-2019-000904-S funded by MCIN/AEI/10.13039/501100011033, as well as a 2022 Leonardo Grant for Researchers and Cultural Creators from the BBVA Foundation.

We wish to thank Antonio Avil\'es, Timur Oikhberg, Mitchell Taylor and Vladimir Troitsky for many helpful discussions and for sharing ideas relevant to the content of this paper.


\begin{thebibliography}{10}

\bibitem{alb-kal}
F.~Albiac and N.J.\ Kalton, \emph{Topics in {B}anach space theory.} Graduate Texts in Mathematics~233, Springer, New York, 2006.
 

\bibitem{AB}
C.D.\ Aliprantis and O.~Burkinshaw, \emph{Positive operators.} Springer, Dordrecht, 2006 (reprint of the 1985 original).

\bibitem{AMR:20}
A.~Avil\'es, G.~Mart\'inez-Cervantes, and J.D.~Rodr\'iguez Abell\'an, \textit{On projective Banach lattices of the form $C(K)$ and $\fbl[E]$.} J.\ Math.\ Anal.\ Appl.\ \textbf{9} (2020), no.\ 1, 124129.

 \bibitem{AMR:21}
A.~Avil\'es, G.~Mart\'{i}nez-Cervantes and J.D.~Rodr\'iguez-Abell\'an, \textit{On the Banach lattice $c_0$.}  Rev.\ Mat.\ Complut.\ \textbf{34} (2021), no.\ 1, 203--213.

\bibitem{APR}
A.~Avil\'es, G.~Plebanek and J.D.~Rodr\'iguez-Abell\'an, \textit{Chain conditions in free Banach lattices.} J.\ Math.\ Anal.\ Appl.\ \textbf{465} (2018), no.\ 2, 1223--1229.

\bibitem{ART}
A.~Avil\'es, J.~Rodr\'{\i}guez and P.~Tradacete, \textit{The free Banach lattice generated by a Banach space.} J.\ Funct.\ Anal.\ \textbf{274} (2018), no.\ 10, 2955--2977.

\bibitem{AT} A.~Avil\'es and P.~Tradacete, \textit{Amalgamation and injectivity in Banach lattices.}  Int.\ Math.\ Res.\ Not.\ IMRN 2023, no.\ 2, 956--997. 

\bibitem{ABT} Y.~Azzouzi, A.~Ben Rjeb, P.~Tradacete, \textit{The strong Nakano property in Banach lattices.} Preprint available at \url{https://arxiv.org/pdf/2401.15629}

\bibitem{Baker}
K.A.\ Baker, \textit{Free vector lattices.} Canad.\ J.\ Math.\ \textbf{20} (1968), 58--66.

\bibitem{BL}
Y.\ Benyamini and J.\ Lindenstrauss, \emph{Geometric nonlinear functional analysis,} Vol.\ 1. Amer.\ Math.\ Soc.\ Coll.\ Publ.~48, American Mathematical Society, 2000.

\bibitem{Bleier}
R.D.\ Bleier, \textit{Free vector lattices.} Trans.\ Amer.\ Math.\ Soc.\ \textbf{176} (1973), 73--87.

\bibitem{C-SC} F.~Cabello S\'{a}nchez and J.M.F.~Castillo, \emph{Homological methods in Banach space theory.} Cambr.\ Studies Adv.\ Math.~{203}, Cambridge University Press, 2023. 

\bibitem{Norm-attaining}
   S.~Dantas, G.~Mart\'inez-Cervantes, J.~Rodr\'iguez Abell\'an and A.~Rueda Zoca, \textit{Norm-attaining lattice homomorphisms.} Rev.\ Mat.\ Iberoam.\ \textbf{38} (2022), no.\ 3, 981--1002.

\bibitem{DJT} J.~Diestel, H.~Jarchow and A.~Tonge, \emph{Absolutely summing operators.}  Cambridge Studies in Advanced Mathematics~43, Cambridge University Press, Cambridge, 1995.

\bibitem{DS}
J.~Diestel and J.~Swart, \textit{The Riesz theorem.} Handbook of measure theory, Vol.~I, II, 401--447. North-Holland, Amsterdam, 2002.

\bibitem{NDJTS} N.~Dunford and J.T.~Schwartz, \emph{Linear operators. Part I: General theory.} Interscience Publishers, Inc., New York, 1958. 

\bibitem{fab-ultimo}
M.~Fabian, P.~Habala, P.~H{\'a}jek, V.~Montesinos and V.~Zizler, \emph{Banach space theory: The basis for linear and nonlinear analysis.} CMS Books in Mathematics/Ouvrages de Math\'ematiques de la  SMC, Springer, New York, 2011.

\bibitem{Goullet} A.\ Goullet de Rugy, \textit{La structure id\'eale des M-espaces.} J.\ Math.\ Pures Appl.\ (9) \textbf{51} (1972), 331--373.

\bibitem{gowers} W.T.\ Gowers, \textit{A solution to Banach's hyperplane problem.} Bull.\ London Math.\ Soc.\ \textbf{26} (1994), no.\ 6, 523--530. 

\bibitem{gm1} W.T.~Gowers  and B.~Maurey, \textit{The unconditional basic sequence problem.} {J.~Amer.\ Math.\ Soc.}~\textbf{6} (1993), 851--874.

\bibitem{gm2} W.T.~Gowers and B.~Maurey, \textit{Banach spaces with small spaces of operators.} {Math.\ Ann.}~\textbf{307} (1997), 543--568.

\bibitem{Grothendieck} A.\ Grothendieck, \emph{Sur certaines classes de suites dans les espaces de Banach et le th\'eor\`eme de Dvoretzky--Rogers.} Bol.\ Soc.\ Mat.\ S\~ao Paulo \textbf{8} (1953), 81--110. 

\bibitem{Hajek_Biorth}
P.~H\'{a}jek, V.~Montesinos Santaluc\'{\i}a, J.~Vanderwerff and V.~Zizler, \emph{Biorthogonal systems in Banach spaces.} CMS Books in Mathematics/Ouvrages de Math\'{e}matiques de la SMC, Springer, New York, 2008.

\bibitem{dHT} D.~de~Hevia and P.~Tradacete, \textit{Free complex Banach lattices.} J.\ Funct.\ Anal.\  \textbf{284} (2023), no.~10, 109888.

\bibitem{JLTTT}
H.\ Jard\'on-S\'anchez, N.J.\ Laustsen, M.\ Taylor, P.\ Tradacete and V.G.\ Troitsky, \textit{Free Banach lattices under convexity conditions.} RACSAM \textbf{116} (2022), no.\ 1, paper~15.

\bibitem{LT2} J.~Lindenstrauss and L.~Tzafriri,  \emph{Classical Banach spaces,} Vol.~II.
 Springer-Verlag, Berlin, 1979.

\bibitem{MN}
P.\ Meyer-Nieberg, \emph{Banach lattices.} Springer, 1991.

\bibitem{OTTT}
T.\ Oikhberg, M.\ Taylor, P.\ Tradacete and V.G.\ Troitsky, \textit{Free Banach lattices.} J. Eur. Math. Soc. (in press).

\bibitem{dePW}
B.~de~Pagter and A.W.\ Wickstead, \textit{Free and projective {B}anach lattices.} Proc.\ Roy.\ Soc.\ Edinb.\ Sect.\ A \textbf{145} (2015), no.~1, 105--143.

\bibitem{Rogalski70}
M.\ Rogalski, \textit{L'espace des fonctions homog\`enes sur un c\^one \`a chapeau universel.} C.~R.~Acad.\ Sci.\ Paris S\'er.~A--B~\textbf{271} (1970), A401--A403.

\bibitem{Rogalski71}
M.~Rogalski, \textit{L'espace des fonctions homog\`enes sur certains c\^ones bien coiff\'es.} Bull.\ Sci.\ Math.~(2)~\textbf{95} (1971), 305--325.

\bibitem{T19} V.G.~Troitsky, \textit{Simple constructions of $\operatorname{FBL}(A)$ and $\operatorname{FBL}[E]$.}  Positivity~\textbf{23} (2019), 1173--1178.

\bibitem{Wickstead}
A.W.\ Wickstead, \textit{Banach lattices with trivial centre.} Proc.\ Roy.\ Irish Acad.\ Sect.\ A \textbf{88} (1988), no.\ 1, 71--83. 

\end{thebibliography}
\end{document}